\documentclass[twoside,10pt, a4paper]{article}
\makeatletter
\def\blfootnote{\gdef\@thefnmark{}\@footnotetext}
\makeatother
\usepackage[ansinew]{inputenc}
\usepackage[T1]{fontenc}
\usepackage{amsmath}
\usepackage{amsthm}
\usepackage{amscd, amsfonts, amssymb}
\usepackage[all]{xy}
\usepackage{mathrsfs}
\usepackage[dvips]{graphicx}
\usepackage{epic, eepic}
\usepackage[colorlinks = true, citecolor = purple, urlcolor  = purple, linkcolor = violet]{hyperref}
\usepackage{xcolor}

\usepackage{titling}
\usepackage{titlesec}
\titleformat*{\section}{\normalfont\Large\color{blue!80!black}}
\titleformat*{\subsection}{\normalfont\large\color{blue!80!black}}
\titleformat*{\subsubsection}{\normalfont\normalsize\color{blue!80!black}}
\usepackage[dotinlabels]{titletoc}
\usepackage{abstract}
\usepackage{helvet}         
\usepackage{courier}        
\usepackage{fix-cm}
\usepackage{BOONDOX-calo}

\numberwithin{equation}{section}

\theoremstyle{definition}
\newtheorem{dfn}{Definition}[section]
\newtheorem{example}{Example}[section]

\theoremstyle{definition}
\newtheorem{remark}{Remark}[section]
\newtheorem{notation}{Notation}[section]
\newtheorem{prop}{Proposition}[section]
\newtheorem{thm}[prop]{Theorem}

\newtheorem*{thm*}{Theorem}

\theoremstyle{theorem}

\renewcommand{\fam}[1]{\mathsf{#1}}

\newcommand{\C}{\mathbb{C}}
\newcommand{\R}{\mathbb{R}}
\newcommand{\Z}{\mathbb{Z}}
\newcommand{\Q}{\mathbb{Q}}

\newcommand{\F}{\mathbb{F}}
\renewcommand{\P}{\mathbb{P}}

\newcommand{\Com}{\mathsf{Com}}
\newcommand{\Mod}{\mathsf{Mod}}
\newcommand{\gMod}{\underline{\mathsf{Mod}}}

\newcommand{\Set}{\mathsf{Set}}
\newcommand{\Art}{\mathsf{Art}}

\newcommand{\azu}{\mathtt{azu}}
\newcommand{\ram}{\mathtt{ram}}

\newcommand{\Simp}{\mathsf{Irr}}
\newcommand{\gSimp}{\mathsf{SMod}}
\newcommand{\gSpec}{\mathbf{Spec}}
\newcommand{\Frac}{\mathrm{Q}}
\newcommand{\Spec}{\mathrm{Spec}}
\newcommand{\Specm}{\mathrm{Specm}}
\newcommand{\Pic}{\mathrm{Pic}}

\newcommand{\Xscr}{\mathbcal{X}}
\newcommand{\Lscr}{\mathcal{L}}

\newcommand{\Max}{\mathrm{Max}}

\newcommand{\oo}{\mathfrak{o}}
\newcommand{\pp}{\mathfrak{p}}
\newcommand{\qq}{\mathfrak{q}}
\newcommand{\pd}{\mathtt{P}}

\newcommand{\dd}{\mathtt{D}}
\newcommand{\ee}{\mathtt{E}}
\newcommand{\ff}{\mathtt{F}}

\newcommand{\af}{\mathfrak{a}}
\newcommand{\mm}{\mathfrak{m}}

\newcommand{\Mm}{\mathfrak{M}}

\newcommand{\eb}{\boldsymbol{e}}
\newcommand{\fb}{\boldsymbol{f}}

\newcommand{\OO}{\mathcal{O}}

\renewcommand{\AA}{\mathcal{A}}
\newcommand{\BB}{\mathcal{B}}
\newcommand{\MM}{\mathcal{M}}

\newcommand{\cent}{\mathfrak{Z}}

\newcommand{\alphabf}{\boldsymbol{\alpha}}

\newcommand{\rhobf}{\boldsymbol{\rho}}
\newcommand{\xibf}{\boldsymbol{\xi}}

\newcommand{\CEnd}{\mathcal{E\!n\!d}}
\newcommand{\CExt}{\mathcal{Ex\!t}}

\newcommand{\Def}{\underline{\mathtt{Def}}}
\newcommand{\TE}{\mathbf{T}}
\newcommand{\TG}{\Gamma}
\newcommand{\pf}{\mathsf{p}}

\newcommand{\ooo}{\hat{\mathfrak{o}}}

\DeclareMathOperator{\id}{id}

\DeclareMathOperator{\rk}{rk}
\DeclareMathOperator{\Ann}{Ann}
\DeclareMathOperator{\Aut}{Aut}
\DeclareMathOperator{\End}{End}
\DeclareMathOperator{\Hom}{Hom}
\DeclareMathOperator{\Ext}{Ext}

\DeclareMathOperator{\im}{im}
\DeclareMathOperator{\Der}{\mathrm{Der}}

\begin{document}
\pretitle{\begin{flushleft}\Large \scshape
} 
\posttitle{\par\end{flushleft}
\rule[8mm]{\textwidth}{0.1mm}
}
\preauthor{\begin{flushleft}\large \scshape
\vspace{-5mm}}
\postauthor{\end{flushleft}\vspace{-8mm}}                                                                                                                                                                                                                                                                                                                                                                                                                                                           
\title{Arithmetic Geometry of non-commutative Spaces with Large Centres}
\author{\large Daniel Larsson }
\date{}

\maketitle
\vspace{-0.5cm}
\begin{abstract}  
\vspace{0.2cm}
\noindent This paper introduces arithmetic geometry for polynomial identity algebras using non-commutative (formal) deformation theory. Since formal deformation theory is inherently local the arithmetic and geometric results that follow give local information that is not visible when looking at the objects from a commutative angle. For instance, it is a precise meaning to be given to two things being ``infinitesimally close'', something being obscured from view when restricting only to a commutative algebraic study. A Platonesque way of looking at this is that the commutative world is a ``shadow'' of a more inclusive non-commutative universe.    

The present paper aims at laying the foundation for further and deeper study of arithmetic and geometry using non-commutative geometry and non-commutative deformation theory. 
\blfootnote{ \textit{e-mail:} \href{mailto:daniel.larsson@usn.no}{ daniel.larsson@usn.no}} 
\end{abstract}
\section{Introduction}
Non-commutative algebra has been present in number theory for a long time (e.g., quaternion algebras and central simple algebras) and as such the approach in this paper is definitely not novel. The novelty presented comes from the use of \emph{non-commutative} deformation theory introduced by O.A. Laudal in \cite{Laudal_def}. 

In addition, deformation theory in arithmetic is also not a novelty. For instance, deformations of Galois groups and deformations of Galois covers of curves in characteristic $p$ are but two examples of prominence. Deformation theory brings out, by its very definition, ``local'' information and the extension to non-commutative bases further brings out local information that is not visible to a commutative eye. We will see several examples of this later. In fact, ``locality'' is a fundamental aspect of both arithmetic and geometry. 

There are not many papers that are dealing with non-commutative algebraic geometry in the context of arithmetic geometry as far as I'm aware. There is T. Borek's version of non-commutative Arakelov theory \cite{Borek, Borek2} and the recent paper \cite{ChanIngalls_ord_arit} by D. Chan and C. Ingalls. Both of these papers define non-commutative algebraic spaces (schemes) differently than what I do. In fact, the starting point is M. Artin and J. Zhang's notion of ``non-commutative projective schemes''. This approach is global by its very definition. On the other hand, Chan--Ingalls use local (\'etale) information coming from orders in central simple algebras over the function field of a commutative scheme. These orders are then the non-commutative schemes (in the sense of Artin--Zhang). However, none of the papers \cite{Borek, Borek2, ChanIngalls_ord_arit} use non-commutative deformation theory.  

\begin{remark}I would be remiss if I did not remark that A. Connes and his collaborators M. Marcolli and K. Consani also studies arithmetic in the context of non-commutative geometry. However, this version is not ``algebraic''. See Connes' webpage for more information. 
\end{remark}

\subsection{The idea}
Let me illustrate the basic idea with the following (not so) hypothetical situation. Let $X$ be a scheme over a field $k$ and let $\mathfrak{G}$ be a group acting on $X$. In general, the quotient $X/\mathfrak{G}$ does not exist as a scheme (but almost always as an algebraic stack) unless, for instance, $X$ is quasi-projective and $\mathfrak{G}$ finite. However, even for finite quotients of quasi-projective schemes, it is sometimes beneficial to introduce stack structures (e.g., at singularities or branch points) on $X/\mathfrak{G}$. For instance,  if $X$ is a curve, the result is often called a ``stacky curve''. Assume for simplicity that $X=\Spec(A)$ for some $k$-algebra $A$ and that $\mathfrak{G}$ is finite. 

Let $k\subseteq k'$. Then there is a one-to-one correspondence between the orbits of $X(k')$ under $\mathfrak{G}$ and the elements of $\Spec(A^\mathfrak{G})(k')$. Without being precise and formal in the least, let $\pf$ be a $k'$-point on $X$ and let $\boldsymbol{orb}(\pf)=\Spec(A/I)$. Then $A/I$ is both a $\mathfrak{G}$-module and an $A$-module; hence an $A[\mathfrak{G}]$-module. Therefore, we have an identification of simple $A[\mathfrak{G}]$-modules over $k'$ and points in $\Spec(A^\mathfrak{G})(k')$, i.e., the orbits. Observe that $A[\mathfrak{G}]$ is a commutative ring. So far, nothing remarkable. 

However, instead of looking at $A[\mathfrak{G}]$ we can look at the crossed product algebra (or skew group ring) $A\langle \mathfrak{G}\rangle$ (sometimes denoted $A\ast \mathfrak{G}$ or $A\# \mathfrak{G}$) which, as an abelian group, is the same as $A[\mathfrak{G}]$ but with multiplication defined by $\sigma x=\sigma(x)\sigma$, for $\sigma\in \mathfrak{G}$ and $x\in A$. This is obviously a non-commutative ring. It is quite easy to see that there is still a one-to-one correspondence between points of $\Spec(A^\mathfrak{G})(k')$ and simple $A\langle \mathfrak{G}\rangle$-modules over $k'$. Put, for simplicity, $B:=A\langle \mathfrak{G}\rangle$.

Now, points on $X(k')$ with non-trivial stabiliser (i.e., points where $\sigma x=x$ for all $\sigma\in \mathfrak{h}\subseteq \mathfrak{G}$; the group $\mathfrak{h}$ is the \emph{stabiliser group} or \emph{isotropy group}) are points of special interest since this is where possible ``stacky-ness'' comes in. In fact, these are the ramification points of the cover $X\twoheadrightarrow X/\mathfrak{G}$. It is a classical topic to study these ramification points and a lot of very interesting things can happen at these points. 

There is a canonical isomorphism between the tangent space at a point $\pf\in X/\mathfrak{G}(k')$ and the vector space $\Ext^1_{B}(M, M)$, where $M$ is the, to $\pf$, associated $A\langle \mathfrak{G}\rangle$-module. In particular, if $\dim(\Ext^1_{B}(M, M))=\dim(X)=\dim(X/\mathfrak{G})$ if $\pf$ is a non-singular point. If $\pf$ is singular
$$\dim(X)=\dim(X/\mathfrak{\mathfrak{G}})\leq\dim(\Ext^1_{B}(M, M)).$$However, when $A$ is non-commutative the $\Ext$-dimension can be greater than what can be expected just by looking from a commutative perspective. In other words, we can have
$$\dim(\Ext^1_{A[\mathfrak{G}]}(M, M))<\dim(\Ext^1_{A\langle \mathfrak{G}\rangle}(M, M)).$$This phenomenon appears only at points of ramification (but not always). The dimension is maximal when the ramification is wild. 

The story does not end there. Assume that $\pf$ and $\mathsf{q}$ are two ramification points on $X/\mathfrak{G}$ with associated $B$-modules $M$ and $N$. Then, as modules over $A[\mathfrak{G}]$ we have $\Ext^1_{A[\mathfrak{G}]}(M, N)=0$. However, as modules over $B=A\langle \mathfrak{G}\rangle$ we always have $\Ext^1_B(M, N)\neq 0$. The interpretation here is that the points $\pf$ and $\mathsf{q}$ lie ``infinitesimally close'' and that $\Ext^1_{B}(M, N)$ is the ``tangent space between $M$ and $N$''. This is not symmetric, we might have 
$$\dim(\Ext^1_{B}(M, N))\neq \dim(\Ext^1_{B}(N, M)).$$ In other words, ``closeness cares about direction''; $\pf$ can be ``closer'' to $\mathsf{q}$ than $\mathsf{q}$ is to $\pf$.  We will see an example of this in the last section. 

Therefore, we can view $\Ext$ as a measure of how ``ramified'' something is, a statement that will be made more precise later in the paper. 

The ring $B=A\langle \mathfrak{G}\rangle$ is an example of a \emph{non-commutative crepant resolution} of $A^\mathfrak{G}$ (see for instance, \cite{StaffordvandenBergh}). Namely, even if $X/\mathfrak{G}=\Spec(A^\mathfrak{G})$ have singularities, the ring $B$ has natural regularity properties as a non-commutative ring (see section \ref{sec:noncomquotient}). Also, T. Stafford and M. van den Bergh proves in \cite{StaffordvandenBergh} that if $A^\mathfrak{G}$ has a non-commutative \emph{crepant} resolution, $\Spec(A^\mathfrak{G})$ has rational singularities. Regardless of the crepant-ness, it is natural to view $B$ as a non-commutative resolution of $\Spec(A^\mathfrak{G})$ as $B$ in any case retains many regularity properties.

\subsection{Enter deformation theory}
Let $\Mod(A)$ be the category of left $A$-modules, $M\in\Mod(A)$ and let $\mathrm{Def}_M$ be a deformation functor of $M$ from the category $\mathsf{C}$ of local, complete, artinian rings to the category of sets. Then, under some mild conditions (see section \ref{sec:deformation}), $\mathrm{Def}_M$ has a pro-representing hull $\hat H$ that is constructed using the tangent space $\Ext^1_A(M, M)$ and matric Massey products (see section \ref{sec:deformation} or \cite{EriksenLaudalSiqveland}, for more details). In other words, $\hat H$ is the completion of the local ring of a moduli space of $A$-modules. Of course, such a moduli space might not exist. However, $\hat H$ always does. 

Let $\fam{M}:=\{M_1, M_2, \dots, M_n\}$ be a family of $A$-modules. Extending the base category $\mathsf{C}$ to non-commutative rings we can define a deformation functor $\mathrm{Def}_\fam{M}^\mathrm{nc}$ of the \emph{family} $\fam{M}$ encoding the \emph{simultaneous} deformations of the modules $M_i$. This functor also have a pro-representing hull, but now this is a matrix ring $(\hat H_{ij})$ with entries being quotients of non-commutative formal power series rings along the diagonal, and bi-modules off-diagonal. The diagonal comes from the spaces $\Ext^1_A(M_i, M_i)$ and the entries off the diagonal come from $\Ext_A^1(M_i, M_j)$, $i\neq j$. As in the commutative case in the previous paragraph, there is an algorithm to compute $(\hat H_{ij})$ (once again, see \cite{EriksenLaudalSiqveland}).

From $(\hat H_{ij})$ one constructs the versal family
$$\hat\rho: A\to \hat\OO_\fam{M},$$with $\hat\OO_\fam{M}$ the matrix ring
$$\hat\OO_\fam{M}:=\big(\Hom_k(M_i, M_j)\otimes_k \hat H_{ij}\big).$$ For more details see section \ref{sec:deformation}.

It is, ultimately, the ring $\hat\OO_\fam{M}$ that is the local object that we're after and that includes all the non-commutative information. Unfortunately, this ring is almost always extremely difficult to compute explicitly. We will see a couple of examples where it is actually possible. Easier, but still very hard, is it to compute the tangent spaces $\Ext_A^1(M_i, M_j)$ (which is the first step in computing $\hat\OO$). 

Let us continue the meta-example above. If $X\twoheadrightarrow X/\mathfrak{G}$ is a singular $\mathfrak{G}$-cover with singular point $\pf$. This is also a branch point. Let $M$ be the orbit of $\pf$, considered as an $A\langle\mathfrak{G}\rangle$-module. It turns out that $\hat\OO_M$ captures a lot of the ramification properties of the cover $X/\mathfrak{G}$ at $\pf$. For instance, in all examples I know of, the tangent space $\hat\OO_M$ is significantly different when the ramification is wild. I'm convinced that this is a general phenomenon but I've not studied it to the point where I can come up with a specific result, other than in special cases, or make a general conjecture. The examples that are computed in the paper clearly show this phenomenon. 

As a consequence, the ring object $\hat\OO_M$ will almost certainly include information concerning wild ramification that is not visible through commutative means. In fact, wild ramification of quotient singularities of arithmetic surfaces has attracted some interest in the last decade (see for instance \cite{ItoSchroer, Kiraly, Lorenzini_wild_surface}) and it is my hope that the introduction of $\hat\OO$ will shed new light on wildly ramified singularities. This seems quite natural since $A\langle\mathfrak{G}\rangle$ can be viewed as a non-commutative resolution\footnote{These are almost certainly not \emph{crepant} resolutions. There seems to be issues concerning crepant resolutions (even in the commutative case) in mixed characteristic.} of $A^\mathfrak{G}$.

\subsection{Polynomial identity algebras}
Let $R$ be a commutative ring and let $P(\mathbf{x})$ be an element in the free algebra $ \Z\langle \mathbf{x}\rangle = \Z\langle x_1, x_2, \dots, x_n\rangle$. A \emph{polynomial identity algebra} (\emph{PI-algebra}) is an $R$-algebra $A$, if there is an $n$ and a $P(\mathbf{x})\in\Z\langle \mathbf{x}\rangle$ as above with $P(a_1, a_2, \dots, a_n)=0$ for all $n$-tuples $(a_1, a_2, \dots, a_n)\in A^n$. 

Clearly, commutative algebras are polynomial identity algebras. Other, less trivial examples include: Azumaya algebras (therefore, also central simple algebras), orders in Azumaya algebras and algebras that are finite as modules over their centres. In particular, $A\langle \mathfrak{G}\rangle$ when $\mathfrak{G}$ is finite, is finite as a module over its centre $A^\mathfrak{G}$ and hence a PI-algebra. Therefore, non-commutative resolutions of singularities as (loosely defined) above are PI-algebras. 

Polynomial identity algebras enjoy some remarkable properties and are very similar to commutative algebras. PI-algebras where extensively studied in the 70's and 80's and as a consequence a lot is known about these algebras. For instance, many of these algebras have good homological properties (regularity, Cohen--Macaulay-ness, e.t.c.). It is impossible to give an exhaustive list of papers dealing with these things but the (quite old) \cite{StaffordZhang} might give some insight. A more modern perspective can be found in \cite{LeBruyn} where PI-algebras are studied using quiver techniques. In fact, the some of the techniques used in \cite{LeBruyn} can probably be adapted for use in studying the ring object $\hat\OO$.   

Since this paper deals with algebras with ``large centres'', meaning exactly that the algebras are finite over its centre, all the machinery of PI-algebras are at our disposal. Of course we will only use a (very!) small part of that machinery. The case where the centre is not ``large'', or more generally, when the algebras are not PI-algebras, is certainly very interesting, but also significantly more difficult to study from an arithmetic-geometric perspective. This is, however, something that should be investigated in the future. 
 
Let $A$ be a PI-algebra with centre $\cent(A)$ and let $\mm$ be a maximal ideal in $\cent(A)$. Then the fibre $A\otimes_{\cent(A)}k(\mm)$ over $\mm$ is either a central simple algebra, or not. The subset of all $\mm$ where $A\otimes_{\cent(A)}k(\mm)$ is a central simple algebra is called the \emph{Azumaya locus}, $\azu(A)$, and its complement, $\ram(A)$, is the support of a Cartier divisor, \emph{ramification locus}. 

The point is the following. Let $\mm\in\azu(A)$. Then $\mm$ is still maximal as an ideal in $A$. However, if $\mm\in\ram(A)$, then $\mm$ splits into disjoint maximal ideals $\mm_i$ inside $A$. In addition, every ideal in $A$ intersects the centre in a unique ideal and, furthermore, if the ideal is maximal so is the intersected ideal. Hence, over $\azu(A)$ there is a one-to-one correspondence between maximal ideals in $\cent(A)$ and maximal ideals in $A$. As a consequence, over $\azu(A)$, $A$ is geometrically the ``same'' as $\cent(A)$.

This indicates that the interesting part of $\cent(A)$ is the ramification locus, and this is indeed the case. Let $\mm\in\ram(A)$. Then, for some $n\geq 2$, $\mm=\mm_1\mm_2\cdots\mm_n$ as ideals in $A$. Let $M_i:=A/\mm_i$, considered as left $A$-modules. Then $\Ext_A^1(M_i, M_j)\neq 0$. In fact, we have the equivalence
$$\Ext^1_A(M_i, M_j)\neq 0\iff \mm_i\cap\cent(A)=\mm_j\cap \cent(A).$$This is the content of M\"uller's theorem (see section \ref{sec:PI}).  Expressed differently, in geometric terms, the points in $A$ over a point $\mm\in\ram(A)$ are infinitesimally close. Here is then where the ring object $\hat\OO$ enters non-trivially.  

Allowing myself to make a definite proposition, the way to think about the above is to view $A$ as a ``non-commutative thickening'' of $\cent(A)$ or, as suggested in the abstract, that $\cent(A)$ is a commutative ``shadow'' of $A$, referring to Plato's cave allegory.

\subsection{Organisation}
The paper is organised as follows. Section \ref{sec:naive_nc_simp} introduces the a first tentative notion of non-commutative space that we will use. This will be definition will expanded in section \ref{sec:nc_alg_space}. Section \ref{sec:deformation} gives a short survey on non-commutative deformation theory. This section can probably be skimmed and referred to as needed. It includes the definition of the ring objects $\hat\OO$. 

The next section, section \ref{sec:nc_alg_space}, defines non-commutative algebraic spaces, which will be the main underlying object in all that follows. Following this is the important section \ref{sec:points} that introduces rational points on non-commutative algebraic spaces. As said above, PI-algebras play a central role in this paper and sections \ref{sec:PI} and \ref{sec:RationalPI} discusses the non-commutative geometry of these algebras, including rational points and tangent spaces. In addition we define what we will mean by an invertible module on a non-commutative algebraic space. 

The next part of the paper is devoted to non-commutative Diophantine Geometry. This section starts of with a recollection of height functions and then we define three types of non-commutative versions: one ``central'', one that is representation theoretic and one that is ``infinitesimal'', taking into consideration the tangent structures of the rational points. 

Section \ref{sec:arithPI-alg} starts by extending the algebras and spaces to ``infinite fibres'', mimicking the technique used in Arakelov geometry. We will not develop a fully fledged Arakelov theory involving metrised Hermitian vector bundles and Chow groups e.t.c.. However, the extension to infinite fibres has the great benefit of including the geometric and arithmetic properties into one coherent object. In this way, the dual nature (geometric/arithmetic) of PI-algebras becomes immediately visible. This section introduces, line sheaves, Cartier and Weil divisors and a rudimentary intersection theory of divisors. This intersection theory is truly noncommutative, and so is not easily computable. Still, it is, in my opinion, a quite natural extension to non-commutative algebraic spaces with large centres.

Finally, in section \ref{sec:example} we look at three examples. The first is the quotient $\mathbb{A}^2/\mu_3$ over an order $\Z[\zeta_3]\subseteq\oo$ in a number field defined by $x\mapsto \zeta_3 x$ and $y\mapsto \zeta_3 y$, where $\zeta_3$ is a third root of unity. As is well-known 
$$\mathbb{A}^2/\mu_3\simeq \Spec\left(\frac{\oo[u, v, w]}{(w^3-uv)}\right),$$at least over fibres where $3$ is invertible. We compute the tangent spaces over all ramification points and compute $\hat\OO$, at least up to obstructions of order two and we give the non-commutative thickening of $\mathbb{A}^2/\mu_3$. We also discuss some arithmetic properties of this thickening such as divisors and heights. As mentioned before, computations here are difficult so, as to not overstate this paper's importance and let it expand beyond reasonable bounds, more difficult computations will have to wait for another time. 

The second example is a family of degree-two thickenings of the integers. We look at the degree-two cover $\Z[\sqrt{d}]/\Z$, where $d\equiv 2, 3\,(\mathrm{mod} 4)$ (and where the integer $d$ is assumed square-free). The quotient of $\Spec(\Z[\sqrt{d}])$ by its Galois group $\mathfrak{G}=\Z/2$ is $\Spec(\Z)$ and so the orbits are in one-to-one correspondence with primes in $\Z$. Therefore, it is reasonable to look at the simple $\Z[\sqrt{d}]\langle \mathfrak{G}\rangle$-modules, which are also in a one-to-one correspondence with $\Spec(\Z)$.  From this we can construct a family of non-commutative spaces, parametrised by $d$, which can be viewed as non-commutative thickenings of $\Z$. In fact, the tangent structure is trivial over all unramified points, but over the ramified ones the ring object $\hat\OO$ is non-trivial. In addition, in the case where the ramification is wild (i.e., at the prime 2), the ring is also obstructed in the sense that the underlying deformations are obstructed.  

In the last example we consider an order over an arithmetic surface. This is algebra is not constructed as a quotient of a scheme by a group. Instead, this is constructed by considering a quantum plane over the surface and then factoring out by an ideal. The resulting tangent structure turns out to be quite interesting. Indeed, we are in the situation alluded to above, where a point ``$\pf$ is closer to $\mathsf{q}$ than $\mathsf{q}$ is to $\pf$''. The way this phenomenon manifests itself is that, in this case, $\Ext^1_{A}(\pf, \mathsf{q})=k^2$, but $\Ext^1_{A}(\mathsf{q}, \pf)=k$. An arithmetic study of this situation is very interesting and is probably worth a paper of its own. 

\subsubsection*{Two final remarks}
\begin{itemize}
	\item I have made the conscious decision to not be consistent in notation and language. The biggest crime is in using the word \emph{centre} of an algebra, both as an algebra itself, but also as a commutative scheme. Therefore, the following typical phrase ``let $\mm$ be a point on $\cent(A)$'' will appear frequently. But not only that: we will use ``maximal ideal'' and ``point'' interchangeably. Later we will also view structure morphisms of representations as ``points''. I will make the brazen assumption that this will not cause the reader too much headache. 
	\item We will switch between global constructions and affine construction rather freely. Where the global situation (i.e., where the base is a scheme) is not a hindrance we will use this. However, at certain points, using algebras over general base schemes, is notationally unwieldy and often obscures the underlying idea by introducing unnecessary complexity in language. In those cases, we will unabashedly work over affine patches. I'm reasonably certain that everything can be globalised straightforwardly, or at least without too much effort.
\end{itemize}

\subsubsection*{Acknowledgements} 

I would like to thank (as always) Arnfinn Laudal who has been constantly encouraging me through many years, being a formidable adversary in the game of enduring my wild ideas, even when he didn't find them particularly interesting. His advice and stamina have been tremendous assets to me.  In addition, I want to thank Arvid Siqveland for helping me make sense of the $\Ext$-computations in section \ref{sec:z3-quotient}. 

\subsection*{Notation}
I will adhere to the following notation throughout.
\begin{itemize}
	\item For a commutative algebra $A$ we denote $\Com(A)$ the category of commutative $A$-algebras. 
	\item For a general algebra (not necessarily commutative) $\Mod(A)$ denotes the groupoid of left $A$-modules up to isomorphism.
	\item The word ``ideal'' always means \emph{2-sided} ideal.
	\item The notation $\Max(A)$ denotes the \emph{set} of (2-sided) maximal ideals, while $\Specm(A)$ denotes the maximal spectrum of $A$, i.e., $\Max(A)$ together with the Zariski--Jacobson topology (see section \ref{sec:naive_nc}). 
	\item  All modules are \emph{left} modules unless otherwise explicitly specified.
	\item $\cent(A)$ denotes the centre of $A$. 
	\item For $\pp$ a prime in $A$, $k(\pp)$ denotes the residue class field of $\pp$. 
	\item We will often identify $$\mm\in\Max(A)\longleftrightarrow k(\mm)\longleftrightarrow\ker\rho\longleftrightarrow M,$$ where $\rho$ is the structure morphism $\rho: A\to \End_k(M)$. 
	\item Abelian sheaves are denoted with scripted letters.
	\item All schemes and algebras are noetherian. Schemes are also assumed to be separated. 
\end{itemize}

\section{Non-commutative algebraic spaces}\label{sec:naive_nc}
\subsection{Non-commutative spaces}\label{sec:naive_nc_simp}

Let $X$ be a scheme and let $\AA$ be a coherent $\OO_X$-algebra, with $\OO_X\subseteq \cent(\AA)$.
\begin{dfn}\label{def:global_simp}The $\AA$-module $\MM$, with structure morphism $$\boldsymbol{\rho}: \AA\to \CEnd_{\OO_X}(\MM),$$ is \emph{simple} if $\MM$ has no $\AA$-submodules. This implies that $\MM$ is simple on the stalks, i.e., 
$$\MM_\pf:=\MM\otimes_{\OO_X}\OO_{X, \pf}$$ is simple as $\AA_\pf:=\AA\otimes_{\OO_X}\OO_{X, \pf}$-module. From this it follows that the fibres are also simple.  
\end{dfn}
Denote by $\Mod(\AA)$ the \emph{set} (or \emph{groupoid})  of all $\AA$-modules up to isomorphism, and $\Mod_n(\AA)$ the set of rank-$n$ such. In addition, put $\gSimp_n(\AA)$ as the set (or groupoid) of isoclasses of simple $\AA$-modules, locally free of rank $n$ (over $\OO_X$) and 
$$\gSimp(\AA):=\bigcup_n\gSimp_n(\AA),$$ the union over all $n$. Similarly,

\begin{remark}\label{rem:globalsimp}
 By a theorem of M. Artin (later extended by C. Procesi) the set $\Simp_n(A)$ can be endowed with the structure of a \emph{commutative} affine scheme, at least over a field of characteristic zero. 
\end{remark}

There is a natural topology on $\Mod(\AA)$, namely the Zariski--Jacobson topology $T_{\mathrm{ZJ}}$. This is the topology generated by the distinguished opens
\begin{align}\label{eq:module_vs_anni}
\begin{split}
D_f&:=\Big\{\MM\in\Mod(\AA)\,\,\,\big\vert\,\,\, \mathrm{Ann}_f(\MM)= 0, \,\, f\in \AA\Big\}\\
&=\Big\{\MM\in\Mod(\AA)  \,\,\,\big\vert\,\,\, f\notin\mathrm{Ann}_{\AA}(\MM)\Big\} ,
\end{split}
\end{align}where $\mathrm{Ann}_f(\MM)$ is the $f$-annihilator of $\MM$, i.e., the submodule of elements $m\in \MM$ such that $fm = 0$, and $\mathrm{Ann}_\AA(\MM)$ is the $\AA$-annihilator, i.e., the ideal in $\AA$ of all $a\in\AA$ such that $a \MM=0$. Observe that we can, and sometimes will, identify $\MM$ with its annihilator ideal $\mathrm{Ann}_\AA(\MM)$ in $\AA$. In addition, observe that
$$\ker\!\Big(\AA\xrightarrow{\rho}\CEnd_{\OO_X}\!(\MM)\Big)\subseteq \mathrm{Ann}_\AA(\MM),$$where $\rho$ is the structure morphism of the $\AA$-module $\MM$. 

We have that $D_f\cap D_g=D_{fg}$ and $D_f\cap D_g=D_g\cap D_f$, so $D_{fg}=D_{gf}$. When $\AA$ is commutative we get back the Zariski topology on $\gSpec(\AA)$ (the global spectrum).  

\begin{remark}\label{rem:rank_module_X}
If $X$ is an $S$-scheme, we will assume that all $\AA$-modules are of finite rank as $\OO_S$-modules. In other words, if $f: X\to S$ is the structure morphism, we assume that $f_\ast\MM$ is a locally free $f_\ast\AA$-module of finite rank on $S$. So, for instance, if $X=\Spec(B)$, $S=\Spec(K)$ (for $K$ a field) and $A$ a $B$-algebra, then an $A$-module $M$ over $X=\Spec(B)$, needs to be a finite-dimensional $K$-vector space and where $A$ acts on $M$ as a $K$-algebra. 
\end{remark}

Now, let 
$$\mathsf{M}:=\big\{\MM_1, \MM_2, \dots, \MM_\ell\big\}$$ be a family of coherent $\AA$-modules such that 
$$\mathrm{rk}_{\OO_X}\!\big(\CExt^1_\AA(\MM_i, \MM_j)\big)<\infty,\quad\text{for all $1\leq i, j\leq \ell$}.$$ The family $\mathsf{M}$ forms a finite set of vertices $$\big\{V_i\mid 1\leq i\leq \ell, \,\,\, \MM_i\longleftrightarrow V_i\big\}$$ in a family of graphs $\TG^m_{\mathsf{M}}$, in which there are $n$ directed edges from  $V_i$ to $V_j$ if $$\mathrm{rk}_{\OO_X}\!\big(\CExt^m_\AA(\MM_i, \MM_j)\big)=n.$$ Observe that there is not necessarily any symmetry in shifting places of $\MM_i$ and $\MM_j$.  We also form the (disjoint) union 
$$\TG^\bullet_\mathsf{M}:=\bigsqcup_{m\geq 0} \TG^m_\mathsf{M}$$ calling this the \emph{augmented tangent space graph at $\mathsf{M}$}, and where $\TG^m_\mathsf{M}$ is the \emph{$m$-th layer} of $\TG^\bullet_\mathsf{M}$. The first layer is called the \emph{tangent space graph} of $\mathsf{M}$.

 In addition, we say that the space $(\TE_{\fam{M}}^m)_{ij}:=\CExt^m_\AA(\MM_i,\MM_j)$ is the \emph{stalk} of $\TG^m_\mathsf{M}$ at $(\MM_i,\MM_j)$.  The set of stalks of $\TG_\mathsf{M}^m$, which we denote $\TE_{\mathsf{M}}^m$, is the \emph{$m$-th layer tangent space}; when $m=1$, we simply say the \emph{tangent space} of $\mathsf{M}$, denoted $\TE_\mathsf{M}$. The total set of stalks is, for obvious reasons, denoted $\TE^\bullet_\mathsf{M}$, and called the \emph{augmented tangent space} of $\mathsf{M}$.    
 
Finally, let $S/R$ be an arbitrary ring extension of commutative rings and let $M$ and $N$ be free of finite rank over $R$. Then the isomorphism (e.g., \cite[Lemma 7.4]{Lam})
$$\Hom_{A\otimes_R S}(M\otimes_R S, N\otimes_R S)\simeq \Hom_A(M,N)\otimes_R S,$$ implies that 
$$\Ext^\bullet_{A\otimes_R S}(M\otimes_R S, N\otimes_R S)\simeq \Ext^\bullet_A(M,N)\otimes_R S.$$From this follows that 
$$\rk_S\!\big(\Ext^\bullet_{A\otimes_R S}(M\otimes_R S, N\otimes_R S)\big)=\rk_R\!\big(\Ext^\bullet_A(M,N)\big),$$so the dimension of $\Ext^1_A(M,N)$ is constant for base extensions. Therefore, the dimensions in the augmented tangent space $\TE^\bullet_{(M,N)}$ are also constant under base change. Clearly, this globalises. 
\subsection{Non-commutative deformation theory of modules}\label{sec:deformation}

Let $A$ be a (not necessarily commutative) $k$-algebra, where $k$ is a commutative ring, and let $\Mod(A)$ be a $k$-linear abelian category of \emph{left} $A$-modules. We recall that $k$-linear means that every $\Hom$-set is a $k$-module. 

\begin{dfn}Let $\Lambda$ be a $k$-algebra. The category $\Mod(A)_\Lambda$ of \emph{right $\Lambda$-objects} is the category of pairs $(X, \rho)$ where $X\in \Mod(A)$ and $\rho: \Lambda\to \End(X)$ and where $\rho(\Lambda)$ acts on the \textit{right} on $X$. The morphisms are the obvious commutative diagrams. 

The $\Lambda$-object $(X, \rho)$ is \emph{$\Lambda$-flat} if the functor
$$X\otimes_\Lambda - :\Mod(\Lambda)\to \Mod(A)_\Lambda$$ is exact.  
\end{dfn} 

Let $\mathsf{art}_r$, $r>0$, be the category whose objects are morphisms 
$$k^r\to \Lambda\xrightarrow{\alpha}k^r, \quad \Lambda\in \Art(k),$$ such that the composition is the identity on $k^r$ and such that $J:=\ker(\alpha)$ is nilpotent. Morphisms are the obvious commutative diagrams. 

If $\{e_1, \dots, e_r\}$ are the idempotents of $k^r$ then we put $\Lambda_{ij}:=e_i\Lambda e_j$. The diagonal consists of subalgebras of $\Lambda$ and the entries off the diagonal are $\Lambda$-bimodules. Notice that $\alpha(\Lambda_{ij})=\delta_{ij}k$ (Kronecker's $\delta$-function). 

We define $\widehat{\mathsf{art}}_r$ to be the category of $r$-pointed pro-objects of $\mathsf{art}_r$. In other words, an object $S$ in $\widehat{\mathsf{art}}_r$ is a projective limit
$$S=\varprojlim_n S/J^n, \quad S/J^n\in \mathsf{art}_r.$$Here $J$ is the kernel of the morphism $S\to k^r$ ($S$ is by definition $r$-pointed).
\begin{dfn}
Let $\fam{M}:=\{M_1, M_2, \dots, M_r\}$ be a family of (left) $A$-modules.  Put $A_\Lambda:=A\otimes_k \Lambda$. 
\begin{itemize}
	\item[(i)] Then a \emph{lifting of $\fam{M}$ to $(\Lambda, \rho)$} is an $A_\Lambda$-module $\fam{M}_\Lambda$ that is $\Lambda$-flat, i.e., that 
	the functor
	$$\fam{M}_\Lambda\otimes_\Lambda - \,\,:  \Mod(\Lambda)\to \Mod(A)_\Lambda$$ is exact. 
	
	The flatness implies that, if $M_i$ is free of rank $n$ over $k$, then $M_i\otimes_k\Lambda$ is also free of rank $n$ as a (right) $\Lambda$-module. This means that 
$\fam{M}_\Lambda = \fam{M}\otimes_{k} \Lambda$ as (right) $\Lambda$-module and, in addition,
$$\fam{M}_\Lambda = \fam{M}\otimes_{k} \Lambda=\big(M_i\otimes_{k}\Lambda_{ij}\big)=\bigoplus_{i,j=1}^rM_i\otimes_k\Lambda_{ij}.$$

\item[(ii)] We also require that the special fibre is $\fam{M}$, i.e., that there is an isomorphism
$$f_\Lambda: \fam{M}_\Lambda =\fam{M}\otimes_k\Lambda\xrightarrow{\id\otimes\alpha} \fam{M},$$induced from the morphism $\alpha:\Lambda \to k^r$. 
	\item[(iii)] Two liftings $\fam{M}_\Lambda$ and $\fam{M}_\Lambda'$ are isomorphic if there is an isomorphism $$h: \fam{M}_\Lambda \to\fam{M}_\Lambda'$$ of $A_\Lambda$-modules such that 
$$f_\Lambda' = (h\otimes \id)\circ f_\Lambda.$$
	\item[(iv)] There is a \emph{non-commutative deformation functor} 
$$\Def_{\fam{M}}:\,\, \mathsf{art}_r \to \Set, \quad \Lambda\mapsto \Def_{\fam{M}}(\Lambda),$$where
$$\Def_{\fam{M}}(\Lambda):=\Big\{\text{all liftings of $\fam{M}$ to $\Lambda$, up to isomorphism of liftings}\Big\},$$and where $\Def_{\fam{M}}(k^r)=\{\bullet\}$. 
\end{itemize}
Observe that $\Def_{\fam{M}}(\Lambda)$ is a groupoid. 
 \end{dfn}

Let $\Def$ be any deformation functor. Any $\Lambda\in\mathsf{art}_r$ comes with a fixed injection $i: k^r\to \Lambda$ and so $\Def$ determines a unique element $\bullet_\Lambda:=\Def(\bullet)\in\Def(\Lambda)$.

 In addition, any $\lambda\in\Def(\Lambda)$ reduces to $\bullet$ under the surjection $\alpha:\Lambda\to k^r$. Any $\lambda\in\Def(\Lambda)$ is a \emph{lift} of $\bullet$ to $\Lambda$. The \emph{trivial lift} of $\bullet$ is the element $\bullet_\Lambda\in\Def(\Lambda)$.

We can extend the deformation functor from $\mathsf{art}_r$ to $\widehat{\mathsf{art}}_r$ by putting
$$\Def(\hat\Lambda):=\varprojlim_n \Def(\hat\Lambda/J^n), \quad\hat\Lambda\in\widehat{\mathsf{art}}_r.$$ 

A \emph{pro-couple} for $\Def$ is a pair $(\hat H, \xi)$ with $\hat H\in\widehat{\mathsf{art}}_r$ and $\xi\in\Def(\hat H)$. A morphism of pro-couples $(\hat H_1, \xi_1)$ and $(\hat H_2, \xi_2)$ is (obviously) a morphism $f:  \hat H_1\to \hat H_2$ such that $\Def(f)(\xi_1)=\xi_2$. 

Yoneda's lemma in the present context states that 
$$\Hom(\underline{\mathtt{h}}_{\hat H}, \Def)\xrightarrow{\,\,\simeq\,\,}\Def(\hat H), $$ where $\underline{\mathtt{h}}_{\hat H}:=\Hom(\hat H, -)$. Therefore, any $\xi\in\Def(\hat H)$ gives a unique morphism of functors $f_\xi: \underline{\mathtt{h}}_{\hat H}\to \Def$. 

If $f_\xi$ is an isomorphism of functors, then $\Def$ is \emph{pro-representable} by the \emph{universal pro-couple} $(\hat H, \xi)$. This is unique up to unique isomorphism of pro-couples.  

Let $f_\xi: \underline{\mathtt{h}}_{\hat H}\to \Def$ be a morphism of functors satisfying, for any surjective morphism $\Lambda\to\Lambda'$ in $\mathsf{art}_r$, the property that
$$\underline{\mathtt{h}}_{\hat H}(\Lambda)\to \underline{\mathtt{h}}_{\hat H}(\Lambda')\underset{\Def(\Lambda')}{\times}\Def(\Lambda)$$is a surjective morphism of functors. We then say that $(\hat H, \xi)$ is \emph{versal} and that $\hat H$ is a \emph{pro-representing hull} with \emph{versal family}, $\xi$. 

It is worth pointing out that the above works, word for word, in any $k$-linear abelian tensor category.

\begin{thm} Suppose $k$ is a field and let $\fam{M}=\{M_1, M_2, \dots, M_r\}$ be a finite family of $A$-modules, with $\dim_k(M_i)<\infty$ for all $1\leq i\leq r$. Assume also that
$$\dim_k\big(\Ext^i_A(M_i, M_j)\big)<\infty, \quad  i =1, 2.$$Then there is a pro-representable hull $(\hat H_{ij})$ for the deformation functor $\Def_{\fam{M}}$, with versal family 
$$\hat\OO_{\fam{M}}:=\fam{M}\otimes_{k}\hat H=(M_i\otimes_k \hat H_{ij}).$$ The algebra morphism (the reduction to the special fibre)
$$\hat\OO_{\fam{M}}=\fam{M}\otimes_k \hat H\xrightarrow{\id\otimes\alpha}\fam{M}\otimes_{k}k^r=\fam{M}$$(we identify 
$\fam{M}$ and $\fam{M}\otimes_k k^r=\oplus M_i$) is implicit in the construction. 

Furthermore, there is an algorithm that computes the hull using (matric) Massey products.
\end{thm} 

We can re-phrase the construction using the structure morphisms of the modules. We begin by noting the isomorphisms 
\begin{align*}
\End_k(\fam{M})\otimes_k\Lambda\simeq\End_\Lambda(\fam{M}_\Lambda)&\simeq \Big(\Hom_k\!\big(M_i, M_j\otimes_k \Lambda_{ij}\big)\Big)\\
&\simeq \big(\Hom_k(M_i, M_j)\otimes_k \Lambda_{ij}\big).
\end{align*} Now, let 
$$\boldsymbol{\varrho}:=\oplus\varrho_i: A\longrightarrow\bigoplus_{i=1}^r\End_k(M_i)=\End_k(\fam{M})\otimes_k k^r=\End_{k^r}\!\!\big(\fam{M}\otimes_k k^r\big)$$ be the structure morphism of the family $\fam{M}$. Then a \emph{lifting of $\boldsymbol{\varrho}$ to $\Lambda$} is an algebra morphism
$$\boldsymbol{\varrho}_\Lambda: A \longrightarrow\End_k(\fam{M})\otimes_k\Lambda=\big(\Hom_k(M_i, M_j)\otimes_k \Lambda_{ij}\big)$$ such that the diagram
$$\xymatrix{&&\End_k(\fam{M})\otimes_{k}\Lambda\ar[d]^{\id\otimes\alpha}\\
A\ar[urr]^{\boldsymbol{\varrho}_\Lambda}\ar[rr]_{\boldsymbol{\varrho}}&&\bigoplus_{1}^r\End_k(M_i)}
$$commutes. 
The vertical morphism is 
 $$\End_k(\fam{M})\otimes_{k}\Lambda\xrightarrow{\id\otimes\alpha} \End_k(\fam{M})\otimes_k k^r=\bigoplus_{i=1}^r\End_k(M_i).$$
A \emph{deformation} of $\boldsymbol{\varrho}$ is then simply an equivalence class of lifts under equivalence of representations. 

From the viewpoint of structure morphisms, the versal family for the corresponding deformation functor $\Def_{\boldsymbol{\varrho}}$ is the morphism
$$\boldsymbol{\hat\varrho}_{\fam{M}}: A\to  \End_k(\fam{M})\otimes_k \hat H=\big(\Hom_k(M_i,M_j)\otimes_k\hat H_{ij}\big).$$This way of viewing deformation theory of modules (i.e., via structure morphisms) is clearly equivalent to the first (i.e., via the category $\Mod(A)$).  

Put 
$$\hat\OO_{\fam{M}}:=\End_k(\fam{M})\otimes_k \hat H=\big(\Hom_k(M_i,M_j)\otimes_k\hat H_{ij}\big).$$ It is natural to view this as the completed local ring at the family $\fam{M}$. Consequently, any algebraisation $H$ of $\hat H$ gives a ring morphism 
$$\boldsymbol{\varrho}_{\fam{M}}: A\longrightarrow  \End_k(\fam{M})\otimes_k H=\big(\Hom_k(M_i,M_j)\otimes_k H\big)$$ and it is natural to view
$$\OO_{\fam{M}}:=\End_k(\fam{M})\otimes_k H=\big(\Hom_k(M_i,M_j)\otimes_k H\big)$$ as the local ring at $\fam{M}$. Observe that we cannot claim that algebraisations are unique, so $\OO_{\fam{M}}$ is not necessarily uniquely determined by  $\hat\OO_{\fam{M}}$.

For a sub-family $\mathsf{M}_0\subseteq \mathsf{M}$, we get an, up to isomorphism, canonical restriction morphism
$$\OO_\mathsf{M}\xrightarrow{\mathrm{res}} \OO_{\mathsf{M}_0}.$$ This can be used to extend the above definition to infinite families of modules. 

Let $\fam{M}$ be an infinite family of $A$-modules and let $\mathsf{S}$ be the category of simplicial sets, i.e., the category of contravariant functors $F:\Delta\to\Set$, where $\Delta$ denotes the category of simplicies. Put 
$$\boldsymbol{\hat\OO}_{\fam{M}}:= \varprojlim_{\fam{M}_0\subseteq\fam{M}}\hat\OO_{\fam{M}_0}, \quad\text{with versal family}\quad \boldsymbol{\hat\varrho}_{\fam{M}}:=\varprojlim_{\fam{M}_0\subseteq\fam{M}}\boldsymbol{\varrho}_{\fam{M_0}}: A\to \boldsymbol{\hat\OO}_{\fam{M}}.$$ We now define the following formal ring object
$$\boldsymbol{\hat\OO}:=\boldsymbol{\hat\OO}_{\Mod(A)}:= \varprojlim_{\mathsf{S}}\hat\OO_{\mathsf{S}}\qquad\text{with}\qquad \boldsymbol{\hat\varrho}:=\varprojlim_{\mathsf{S}}\boldsymbol{\hat\varrho}_{\mathsf{S}}: A\to \boldsymbol{\hat\OO},$$ and its (possibly non-unique) algebraic ring object
$$\boldsymbol{\OO}:=\boldsymbol{\OO}_{\Mod(A)}:= \varprojlim_{\mathsf{S}}\OO_{\mathsf{S}}\qquad\text{with}\qquad \boldsymbol{\varrho}:=\varprojlim_{\mathsf{S}}\boldsymbol{\varrho}_{\mathsf{S}}: A\to \boldsymbol{\OO}.$$It should be clear what the notation means.

The exact same construction extends to sheaves over a scheme, with the exception that one needs to use a global version of Hochschild cohomology instead of the ordinary affine one (which gives the $\Ext$-groups). See \cite[Sec. 3.4]{EriksenLaudalSiqveland} for details on this. 

Observe that the construction of the pro-representing hull $\hat H$ involves the two first layers of the augmented tangent space $\TE^\bullet_\mathsf{M}$.

\subsection{Non-commutative algebraic spaces}\label{sec:nc_alg_space}
We now return to the global situation. 

\begin{dfn}Let $\AA$ be coherent $\OO_X$-algebras as above. We call $$\gMod(\AA):=(\Mod(\AA), T_\mathrm{ZJ}, \boldsymbol{\OO})$$ the \emph{non-commutative algebraic space} (or \emph{non-commutative scheme}) associated with $\AA$. We also put $$\gMod_{n}\!(\AA):=(\Mod_n\!(\AA), T_\mathrm{ZJ}, \boldsymbol{\OO}_n).$$ Here $\boldsymbol{\OO}_n$ is obviously the restriction of $\boldsymbol{\OO}$ to $\gMod_n(\AA)$. The algebra $\AA$ is to be viewed as the algebra of global sections of $\underline{\mathcal{O}}$, i.e., informally, as
$$\AA=\Gamma(\gMod(\AA), \boldsymbol{\OO})=H^0(\gMod(\AA), \boldsymbol{\OO}).$$The object $\boldsymbol{\OO}$ gives the \emph{local} information of $\gMod(A)$. We call $\boldsymbol{\OO}$ the \emph{structure object} of $\Xscr_\AA$.
\end{dfn}
The local nature of $\boldsymbol{\OO}$ is the reason that we don't need to deform the $\AA$-sheaves in $\fam{M}$ as \emph{sheaves}, but can do this over the affine patches on $X$, ignoring the glueing.  

\begin{dfn}	If $\AA$ is a finite-rank $\OO_X$-algebra, where $X$ is of finite type over an arithmetic ring (by which we mean an order, most often the maximal order, in a number field), $\gMod(\AA)$ is called an \emph{arithmetic space} (or \emph{arithmetic scheme}). However, the objects $\boldsymbol{\OO}$ and $\boldsymbol{\OO}_n$ can only defined fibre-wise since they are local, deformation-theoretic, objects. 
\end{dfn}

\begin{dfn}The space $\underline{\Mod}(\AA)$ is \emph{regular at $\mathsf{M}$} if $\Ext_\AA^2(\MM_i,\MM_j)=0$ for all $\MM_i, \MM_j\in\mathsf{M}$; $\underline{\Mod}(\AA)$ is \emph{regular} if it is regular at all families $\mathsf{M}$.
\end{dfn}
It is worth pointing out that the above is a deformation-theoretic definition of regularity.  Hence a point is regular if all deformations of that point are unobstructed. This is a strong condition. There are other, weaker, notions of non-commutative regularity (e.g., Auslander regularity) that we will encounter later (but won't define formally).
\begin{notation}\label{not:ModX} To simplify notation, we will often use the notation $$\Xscr:=\underline{\Mod}(\AA)$$ or $\Xscr_\AA$ if we need to be explicit concerning what algebra we are working with. This will most often be apparent from the context. Recall that $$\underline{\Mod}(\AA)=(\Mod(\AA), T_{\mathrm{ZJ}}, \boldsymbol{\OO}).$$ However, we will be a bit sloppy and make the identifications
$$\Xscr_\AA\quad\longleftrightarrow\quad\underline{\Mod}(\AA)\quad\longleftrightarrow\quad\Mod(\AA).$$ I don't think this will cause much confusion as we will be specific when we use the topology and $\underline{\OO}$. 
\end{notation}

Finally, the following is important enough to warrant its own remark.
\begin{remark}\label{rem:projectivisation}
Let $\AA$ be a non-commutative algebra over an affine $S$-scheme $X$. Then, taking the projective closure $f: X\hookrightarrow \mathbb{P}_S^n$, we can push-forward $\AA$ to an algebra $f_\ast\AA$ on $\im f$. This can be useful when considering non-commutative algebras that are finite over their centre since we then can use projective techniques (properness in particular) in the study of $\AA$. 
\end{remark}
\subsection{Point modules}\label{sec:points}
It turns out that it is not easy to define closed points for non-commutative spaces over non-algebraically closed fields, in a way generalizing the commutative situation naturally. Since we view closed points as ``local objects'' we discuss the case of $\AA=A$ affine over a field $k$ and then explain how to globalise. 

Let $k$ be a field and $A$ a $k$-algebra. A \emph{point} on $\Xscr_A$ is a finite-dimensional representation $\rho: A\to \End_k(M)$. We will identify $\rho$, the kernel, $\ker\rho$, and the module $M$, referring to all these as the point $\rho$. 

Assume that $\ker\rho = \mm_1\mm_2\cdots\mm_s$, where the $\mm_i$ are, not necessarily distinct, maximal ideals. Then
$$A/\ker\rho=A/\mm_1\times A/\mm_2\times\cdots\times A/\mm_s, $$with each $A/\mm_i$ a simple $k$-algebra. Let $\rho_i$ be the induced morphism $\rho_i: A\to A/\mm_i\hookrightarrow \End_k(M)$.  Now we have
\begin{align*}
E&:=\cent(A/\ker\rho)=\cent(A/\mm_1)\times \cent(A/\mm_2)\times\cdots\times \cent(A/\mm_s)\\
&=k(\mm_1)\times k(\mm_2)\times\cdots\times k(\mm_s)\\
&=k(\rho_1)\times k(\rho_2)\times\cdots\times k(\rho_s)
\\
&=k_1\times k_2\times \cdots\times k_s,
\end{align*}where each factor $k_i=k(\mm_i)=k(\rho_i)$ is a finite field extension of $k$. Therefore, $E$ is an \'etale algebra over $k$.

We now define:
\begin{dfn}Let $\rhobf:A\to \End_k(M)$ be a point on $\Xscr_A$.  
\begin{itemize}
	\item[(a)] Then $\rhobf$ defines a \emph{closed \'etale point} if the $k$-algebras
$\bar\rho_i:=A/\mm_i$ are all simple. We denote the set of all closed \'etale points $\Mod_{\text{\'et}}^\bullet(A)$.
	\item[(b)] If $s=1$, $\rhobf$ is simply a \emph{closed point}, which write $\rho$ (in non-boldface). The set of all closed points is denoted $\Mod^\bullet(A)$.
	\item[(c)] The fields $k_i=k(\mm_i)=k(\rho_i)=\cent(\bar\rho_i)$ are the \emph{residue fields} of $\rhobf$. The $k$-algebra $E$ is the \emph{residue ring} of $\rhobf$. 
	\item[(d)] The algebras $\bar\rho_i$, or equivalently the $\mm_i$, are the \emph{underlying points} of $\rhobf$. 
\end{itemize}
\end{dfn}

\begin{dfn}Let $k$ be a field and let $A$ and $S$ be $k$-algebras.
\begin{itemize}
	\item[(a)] An \emph{\'etale $S$-rational point} on $\Xscr_A$ is an algebra morphism
$$\xibf: A\to S$$ such that 
$$A/\ker\xibf=A/\mm_1\times A/\mm_2\times\cdots\times A/\mm_s$$ is a direct product of prime rings. If the $A/\mm_i$ are artinian, being prime is equivalent to being simple so the $\mm_i$ are maximal ideals. This applies in particular to the case when $S$ is artinian. 
	\item[(b)] The \emph{underlying \'etale point} of $\xibf$ is the set of algebras $\bar\xi_i:=A/\mm_i$. 
	\item[(c)] If all $\bar\xi_i$ are simple, the point is \emph{closed}, otherwise it is \emph{non-closed}. The point is \emph{open} if all the $\bar\xi_i$ are non-simple. 
	\item[(d)] If $s=1$, the map $\xibf$ is an \emph{$S$-rational point}, which we write in non-boldface: $\xi$; the algebra $\bar\xi=A/\mm$ is then the (unique) underlying point of $\xi$.  
	\item[(e)] If $S=\End_L(M)$ for some field $L$ and $M$ finite-dimensional over $L$, we say that $\xi$ is an \emph{$L$-rational point}.
	\item[(f)] $\xibf$ is a \emph{geometric \'etale point} if it is closed and $\cent(\bar\xi_i)=k^\mathrm{al}$ for all $i$.  
\end{itemize}
As usual we denote the $S$-rational points on $\Xscr=\underline{\Mod}(A)$ as $$\Xscr(S)=\underline{\Mod}(A)(S).$$ 
\end{dfn}
\begin{example}Let $A$ and $S$ be $k$-algebras where $k$ is a field and let $\xi: A\to S$ be an $S$-point on $\Xscr_A$. For any field extension $k'$ of $k$ the base extension $\xi_{k'}:=\xi\otimes k'$ defines an $S\otimes_k k'$-rational point. 

In particular, if $S=\End_k(M)$ we have that
$$ \xi_{k'}=\xi\otimes k': A\to \End_k(M)\otimes_k k' =\End_{k'}(M\otimes_k k')$$ is a $k'$-rational point. The point is then geometric if $k'=k^\mathrm{al}$. 
\end{example}
\begin{example}
Let $S$ be a commutative ring and $\rho:\cent(A)\to S$ an $S$-point. Put $\mm_\cent:=\ker \rho$. The extension of $\mm_\cent$ to $A$ defines an ideal, $\mm$, not necessarily maximal. Then 
$$\rhobf: A\to A/\mm$$ defines an \'etale $A/\mm$-rational point.
\end{example}

There is a bijective correspondence
$$\Big\{\text{$S$-rational points, $\xi:A\to S$}\Big\}\quad\longleftrightarrow\quad\Big\{(A/\mm, j)\,\,\big\vert\,\, j:A/\mm\hookrightarrow S\Big\}.$$ Notice that $S$ becomes a $\cent(A/\mm)$-algebra via $j$.

For any extension $S\subset T$ the group $\Aut_S(T)$ acts on $\Xscr_A(S)$ as
$$(A/\mm, j)^\sigma=(A/\mm, \sigma\circ j)$$for all $\sigma\in\Aut_S(T)$.

\subsection{Polynomial identity algebras}\label{sec:PI}
We will be particularly interested in the case where $\AA$ is finite as a module over $\OO_X$. In this case $\AA$ is locally a polynomial identity (PI-) algebra with $\OO_X\subseteq \cent(\AA)$, where $\cent(\AA)$ denotes the centre of $\AA$. 

Let $U\subseteq X$ be an affine open set and put $A:=\AA(U)$. In addition, let $\rho:A \to\End_k(M)$ be a point on $\Xscr_\AA$. The kernel $\Mm:=\ker\rho$ restricts to an ideal $\mm$ in $\cent(A)$. If $\Mm$ is prime (or maximal) then so is $\mm$ (see \cite[III.1.1]{BrownGoodearl}, for instance). Put $Y:=\Spec(\cent(\AA))$. The intersection of ideals in $A$ with the centre defines a finite, surjective, morphism $\Psi: \Xscr_A\to Y$.

Conversely, if $\pp\in Y$ there is a prime $\mathfrak{P}\in\Spec(A)$ such that $\pp=\mathfrak{P}\cap\cent(A)$. The locus in $Y$ where the extension $\mathfrak{P}$ is unique is called the \emph{Azumaya locus}, $\azu(A)$,  of $A$. This is a Zariski open subset and the complement is the support of a Cartier divisor (e.g., \cite[III.2.5]{Jahnel}) called the \emph{ramificication locus}, $\ram(A)$. 

Hence, if $\rho$ is an \'etale point, the intersection $(\ker\rho)\cap\cent(A)$ is a finite collection of closed points in $Y$. 

The above indicates that there is a close relationship between the geometry of $\Xscr_\AA$ and the geometry of $\gSpec(\cent(\AA))$. In the proposition below we will freely use the identification in (\ref{eq:module_vs_anni}) to identify modules with their corresponding annihilator ideals. We don't need to assume that the annihilator ideals are prime.  
\begin{prop}\label{prop:PItopology}Let $\AA$ be a PI-algebra over $\OO_X$, with centre $\cent(\AA)$.  
\begin{itemize}
	\item[(i)] Suppose $D_{f'}$ is a distinguished open set on $\Xscr_\AA$. Then $D_{f'}\cap\gSpec(\cent(\AA))$ is a distinguished open set on $\gSpec(\cent(\AA))$.
	\item[(ii)] Conversely, suppose $D_{f}$ is a distinguished open on $\gSpec(\cent(\AA))$. Then $D_{f'}$, where $f=f'\cap \cent(\AA)$, is a distinguished open in $\Xscr_\AA$.
\end{itemize} 
As a consequence, since the distinguished open sets are basis sets for the Zariski and Zariski--Jacobson topologies on $\gSpec(\cent(\AA))$ and $\Xscr_\AA$, respectively, the Zariski--Jacobson topology on $\Xscr_\AA$ is compatible with the Zariski topology on $\gSpec(\cent(\AA))$. 
\end{prop}

\begin{proof}
	Suppose $\af'\in D_{f'}$. Then $f'\not\in\af'$ and so $f'\cap\cent(\AA)\not\in \af'\cap\cent(\AA)$, in other words, $\af'\cap\cent(\AA)\in D_{f'\cap\cent(\AA)}$. Conversely, suppose that $\af\in D_{f}\subset\gSpec(\cent(\AA))$ and take some $f'\in \AA$ such that $f'\cap \cent(\AA)=f$. The extension of $\af$ to $\AA$ can be split into a number of ideals $\{\af_i'\subset \AA\}$. Suppose $f\in \af_i'$ for some $i$. Then $f'\cap\cent(\AA)\in \af_i'\cap\cent(\AA)\iff f\in \af$, a contradiction. Hence $f'\notin\af_i'$ and so the extension of $D_{f}$ to $\Xscr_\AA$ is a distinguished open set. 
\end{proof}
If we view $\AA$ as a sheaf of algebras over its centre $\gSpec(\cent(\AA))$, we can use central localization, i.e., localization of $\AA$ at multiplicatively closed sets in its centre (such a localization is always well-defined), and find
$$\AA(D_f)=\AA_f\quad\text{and}\quad \Xscr_\AA(D_f)=\Xscr_{\AA_f}.$$Observe that these two statements are different. The first one says that the sheaf $\AA$ is equal to $\AA_f$ over $D_f\subseteq\gSpec(\cent(\AA))$ on the centre, while the other says that the open set $D_f$, viewed as a distinguished open on $\Xscr_\AA$, is equal to $\Xscr_{\AA_f}$. 

However, the proposition says more: we can localize on $\AA$ directly by simply restricting the distinguished open sets $D_{f'}$ of $\Xscr_\AA$ to distinguished opens $D_f$ on $\gSpec(\cent(\AA))$.

There are a number of constructions that follow from proposition \ref{prop:PItopology}. We list the most obvious and important ones below.

\subsubsection{Sheaves}\label{sec:sheaves}Indeed, proposition \ref{prop:PItopology} allows us to define sheaves on $\Xscr_\AA$ with the Zariski--Jacobson topology, as sheaves on  $\gSpec(\cent(\AA))$, defined by the restrictions $D_f\mapsto D_{f\cap\cent(\AA)}$.
	
\begin{dfn}\label{def:A-mod}
An \emph{$\AA$-module} on $\Xscr_\AA$ is a sheaf $\mathcal{F}$ on $\gSpec(\cent(\AA))$ such that each $\mathcal{F}(D_f)=\mathcal{F}(D_{f\cap\cent(\AA)})$ is an $\AA_{f\cap\cent(\AA)}$-module. Notice that this implies that $\mathcal{F}$ is automatically an $\OO_{\cent(\AA)}$-module (we write $\OO_{\cent(\AA)}$ instead of $\OO_{\gSpec(\cent(\AA))}$ to simplify notation).
\end{dfn}

Recall that a \emph{Jacobson ring} is a ring in which every prime ideal is the intersection of primitive ideals. The most important examples are, fields, Dedekind domains with infinitely many prime ideals and affine algebras over Jacobson rings. 

An algebra $A$ is \emph{generically free} if every finitely generated $A$-module $M$ is \emph{centrally} locally free, in the sense that there is an $f\in\cent(A)$ such that $M_f:=M\otimes_A  A_f$ is free as an $A_f:=A\otimes_{\cent(A)}\cent(A)_f$-module. 

It is a fact that every affine PI-algebra over a Jacobson ring is generically free (see \cite[Theorem 6.3.3]{RowenII}). Therefore, for affine PI-algebras $\AA$ over Jacobson schemes, finitely generated $\AA$-modules are locally free. Hence,
\begin{prop}Every finitely generated module over $\Xscr_\AA$, where $\cent(\AA)$ is a Jacobson ring, is locally free as an $\AA$-module.
\end{prop}  

For simplicity of notation, we will temporarily work with affine algebras. An \emph{invertible} $A$-module is a finitely generated, projective, $A$-bimodule such that $$A\simeq \End_A({}_AM)\quad\text{and}\quad A\simeq \End_A(M_A),$$ where we denote the left action of $A$ on $M$ by ${}_AM$ and similarly the right action. 

The set of isomorphism classes of all such form a group, denoted $\Pic(A)$, under tensor products (see \cite{FrohlichPicard}):
$$[M]\cdot [N]:=[M\otimes_A N].$$This is well-defined since $M$ and $N$ are $A$-bimodules. 

Let $R$ be a commutative subring of $A$ and let $M$ be an $A$-bimodule. If $Mr=rM$ for all $r\in R$, we say that $M$ is defined \emph{over} $R$. Notice that this need not be the case in general since $R$ can act differently from the left and from the right. 

We denote the set of invertible $A$-modules over $R$ by $\Pic_R(A)$ (or $\Pic_R(\Xscr_A)$), and by $\Pic^{\text{lf}}_R(A)$ (or $\Pic^{\text{lf}}_R(\Xscr_A)$) the set of invertible $A$-modules that are locally free over $R$.  If $R=\cent(A)$ we put $\Pic_\cent(A):=\Pic_{\cent(A)}(A)$. 

Suppose $\Lscr\in\Pic(R)$. Hence, $\Lscr$ is a locally free $R$-module of rank one. The following is almost certainly well-known.
\begin{prop}The map $T: \Pic(R)\to \Pic_R(A)$ defined by
$$T(\Lscr):=A\otimes_R\Lscr\otimes_R A$$ is a group homomorphism.
\end{prop}
\begin{proof}
	The proposition follows from the following simple (and obvious) computation:
	\begin{align*}
		T(\Lscr_1\otimes_R\Lscr_2)&=A\otimes_R\Lscr_1\otimes_R\Lscr_2\otimes_R A\\
		&=A\otimes_R\Lscr_1 \otimes_R A\otimes_A A \otimes_R\Lscr_2\otimes_R A\\
		&=(A\otimes_R\Lscr_1\otimes_R A)\otimes_A (A\otimes_R\Lscr_2\otimes_R A)\\
		&=T(\Lscr_1)\otimes_A T(\Lscr_2).\qedhere
	\end{align*}
\end{proof}
The above globalises immediately. The following proposition is corollary 4 in \cite{FrohlichPicard}.
\begin{prop}Let $f: R\to S$ be a surjective ring morphism. Then
$$\Pic_R(A)=\Pic_S(A)\quad\text{and}\quad \Aut_R(A)=\Aut_S(A).$$In particular, if $R$ is local and $f$ is the reduction morphism $R\to k(\mm)$, then 
$$\Pic_R(A)=\Pic_{k(\mm)}(A)\quad\text{and}\quad \Aut_R(A)=\Aut_{k(\mm)}(A).$$
This proposition does not globalise easily since the proof uses Morita equivalences and I don't know if Morita theory can be globalised. 
\end{prop}
Recall that $\Psi$ is the morphism defined by contracting ideals from $\AA$ to $\cent(\AA)$. 
\begin{dfn}\label{def:very_ample_sheaf} 
	 An invertible sheaf $\mathcal{L}\in \Pic^\text{lf}_\cent(\AA)$ is \emph{very ample} if 
	  $$\det(\Psi_\ast\mathcal{L})\simeq \mathcal{N}\vert_{\gSpec(\cent(\AA))},$$
	  for $\mathcal{N}$ a very ample invertible sheaf on $\P(\cent(\AA))$. In other words, there is an embedding $\alpha: \P(\cent(\AA))\hookrightarrow \P^m$, for some $m\geq 1$, such that $\mathcal{N}\simeq \alpha^\ast\OO(1)$.
	  
	  The set of all very ample sheaves on $\Xscr_\AA$ is denoted $\Pic^\text{va}_\cent(\Xscr_\AA)$ or simply $\Pic^\text{va}_\cent(\AA)$. 
	\end{dfn}
We also make the following definition:
	\begin{dfn}\label{def:canonical_sheaf} 
	 The sheaf $\mathcal{C}\in \Pic^\text{lf}_\cent(\AA)$ is a \emph{canonical sheaf} if 
	  $$\det(\Psi_\ast\mathcal{C})\simeq \mathcal{C}_\cent,$$
	  for $\mathcal{C}_\cent$ is a canonical sheaf on $\P(\cent(\AA))$. 
	\end{dfn}
There is probably a more general and sophisticated approach to canonical sheaves by using so called ``phase functors'' (see \cite{Larsson_phase} for a discussion).
%

\subsubsection{``Morphisms'' between non-commutative spaces}
Let $\AA$ and $\mathcal{B}$ be two algebras over the same $k$-scheme $X$ (where $k$ is a field), and let $$\psi:\,\,\ \BB\to \AA$$be an algebra morphism. This defines a (set-theoretic) morphism $$\Mod(\psi):\,\,\,\Mod(\AA)\to \Mod(\BB).$$A ``morphism'' (which we will call a ``non-commutative morphism'') between $\Xscr_{\AA}$ and $\Xscr_\BB$ should respect both the Zariski--Jacobson topology and map $\OO_{\Xscr(\BB)}$ to $\OO_{\Xscr(\AA)}$. We now show that this can be done in a formal sense. 
\begin{thm}\label{thm:nc_morphism}Let $\psi$ be as above and let $$\fam{M}:=\big\{\MM_1, \MM_2, \dots, \MM_n\big\}$$be a finite family of $\AA$-modules (over $X$) such that $\dim_k(\MM_i)<\infty$ (cf. remark \ref{rem:rank_module_X}). Then $\psi$ induces a \emph{non-commutative morphism} 
$$\Xscr(\psi):\,\,\, \Xscr_\AA\to \Xscr_\BB$$defined by the composition 
\begin{equation}\label{eq:CompXBA}
\BB\xrightarrow{\psi}\AA\xrightarrow{\varrho}\CEnd_{\OO_X}(\fam{M}).
\end{equation}
\end{thm}
\begin{proof}
	We work over affine patches. It will be clear that everything glues. 
	
	So, let $\boldsymbol{\varrho}=\oplus \varrho: A\to  \End_k(\fam{M})$ be a family of $A$-modules. Then, clearly, $\boldsymbol{\varrho}\circ\psi:B\to\End_k(\fam{M})$ is a family of $B$-modules. 
	
	Deforming $\fam{M}$ as $A$-modules gives the versal family
	$$\boldsymbol{\hat\varrho}:\,\,\, A\longrightarrow \left(\Hom_k(M_i, M_j)\otimes_k \hat H_{ij}\right)$$inducing
	$$\boldsymbol{\hat\varrho}\circ\psi:\,\,\, B\longrightarrow A\longrightarrow \left(\Hom_k(M_i, M_j)\otimes_k \hat H_{ij}\right).$$Deforming $\boldsymbol{\varrho}\circ\psi$ directly gives a morphism
	$$B\longrightarrow  \left(\Hom_k(M_i^B, M_j^B)\otimes_k \hat H(\boldsymbol{\varrho}\circ\psi)_{ij}\right),$$where $M_i^B$ means viewing $M_i$ as $B$-module via $\boldsymbol{\varrho}\circ\psi$. By versality, we get an induced morphism
	$$\hat\psi_{ij}:\,\,\,\left(\Hom_k(M_i^B, M_j^B)\otimes_k \hat H(\boldsymbol{\varrho}\circ\psi)_{ij}\right)\longrightarrow \left(\Hom_k(M_i, M_j)\otimes_k \hat H_{ij}\right),$$giving (by definition)
	$$\hat\psi_{ij}:\,\,\,\left(\Hom_k(M_i^B, M_j^B)\otimes_k \hat H(\boldsymbol{\varrho}\circ\psi)_{ij}\right)\longrightarrow \hat\OO_\fam{M}.$$Putting $$\hat\OO_{\fam{M}^B}:=\left(\Hom_k(M_i^B, M_j^B)\otimes_k \hat H(\boldsymbol{\varrho}\circ\psi)_{ij}\right)$$ we get 
	$$\hat\psi_{ij}:\,\,\,\hat\OO_{\fam{M}^B}\longrightarrow \hat\OO_\fam{M},$$giving, formally, the desired local (algebra) morphism of the ``structure sheaves'' $\OO$.
	
As for the topology, this is essentially obvious. Let $f\in B$ and let $M\in D_f$ with structure morphism $\rho:B\to\End_k(M)$. Hence, by definition, $\rho(f)M\neq 0$.  That $M\in\im\Mod(\psi)$ means that there is a $\sigma: A\to \End_k(M)$ such that $\rho=\sigma\circ\psi$. Then, if $\sigma(\psi(f))M=0$, we would have that $\rho(f)M=0$. Hence, $\Mod(\psi)^{-1}(D_f)=D_{\psi(f)}$, completing the proof that $\Mod(\psi)$ is continuous for the Zariski--Jacobson topology. 
\end{proof}

If $X$ is an $S$-scheme where $S$ is not the spectrum of a field, we get a \emph{topological} map $\Xscr_\AA\to\Xscr_\BB$ from (\ref{eq:CompXBA}). The morphisms $\hat\psi$ between the $\OO$-rings (as in the above proof) can only be constructed \emph{fibre-by-fibre} over $S$.

\subsubsection{Central subschemes and their non-commutative lifts}\label{sec:central_subschemes}

Let $Y:=\gSpec(\cent(\AA))$ and let $W\overset{j}{\hookrightarrow} Y$ be a closed subscheme. Since $\AA$ is a locally free $Y$-algebra, we can restrict $\AA$ to $W$ via $j$. This means that $j^\ast\AA$ is a locally free $W$-algebra with $\OO_W\subseteq\cent(j^\ast \AA)$. 

Algebraically, we have
$$\xymatrix{\AA\ar[r]&\AA\otimes_{\cent(\AA)}\big(\cent(\AA)/\mathcal{I}\big)=\AA/\langle\mathcal{I}\rangle\\
\cent(\AA)\ar@{^(->}[u]\ar@{->>}[r]& \cent(\AA)/\mathcal{I}\ar@{^(->}[u]}$$and geometrically
\begin{equation}\label{eq:subscheme_lift}
\begin{split}
\xymatrix{\Mod(j^\ast\AA)\ar@{..>}[d]\ar[r]&\Mod(\AA)\ar@{..>}[d]\\
	W\ar@{^(->}[r]&Y,}
\end{split}
\end{equation}
where the dotted arrows indicate that the morphisms are defined by restriction of ideals. 
	
In fact we can extend this to include the $\OO$-rings.

\begin{prop}With the notation as above, diagram (\ref{eq:subscheme_lift}) can be completed to the diagram
\begin{equation}\label{eq:subscheme_lift_completed}
\begin{split}
\xymatrix{\Xscr_{j^\ast\AA}\ar@{..>}[d]\ar[r]&\Xscr_\AA\ar@{..>}[d]\\
	W\ar@{^(->}[r]&Y.}
	\end{split}
\end{equation}In this way $\Xscr_{j^\ast \AA}$ defines a \emph{closed non-commutative subspace} of $\Xscr_\AA$. 
\end{prop}
\begin{proof}
The compatibility between the topologies is obvious so we only need to prove that the $\OO$-ring on $\Mod(A)$ restricts to $\Mod(j^\ast A)$.
	
	The argument is essentially the same deformation-theoretic argument used in the proof of theorem \ref{thm:nc_morphism}. We work locally, so let $\boldsymbol{\varrho}:A/I\to\End_k(\fam{M})$ be an \'etale point on $\Mod(j^\ast A)$. This gives the \'etale point $\mathrm{pr}\circ\boldsymbol{\varrho}: A\to\End_k(\fam{M})$ on $\Mod(A)$, where  $\mathrm{pr}:A\to A/I$ is the projection. 
	
	Deforming $\fam{M}$ as an $A/I$-module gives the versal family
	$$\boldsymbol{\hat\varrho}_{A/I}:\,\,\, A/I\longrightarrow \left(\Hom_k(M_i, M_j)\otimes_k \hat H(A/I)_{ij}\right)$$ and as an $A$-module gives the family
	$$\boldsymbol{\hat\varrho}:\,\,\, A\longrightarrow \left(\Hom_k(M_i, M_j)\otimes_k \hat H_{ij}\right)$$By versality we get an algebra morphism 
	$$\left(\Hom_k(M_i, M_j)\otimes_k \hat H_{ij}\right)\longrightarrow \left(\Hom_k(M_i, M_j)\otimes_k \hat H(A/I)_{ij}\right)$$which gives the desired restriction of $\OO$-rings from $\Mod(A)$ to $\Mod(j^\ast A)$. 
\end{proof}

In the case of PI-algebras, the interesting case is when $W\cap\ram(\AA)\neq \emptyset$.

 \subsection{Rational $L$-points of PI-algebras}\label{sec:RationalPI}
Let $R$ be commutative ring and $A$ an $R$-algebra. Put $B:=\cent(A)$ and let $\alpha: B\to L$ be an $L$-rational point, where $L$ is a field. Let $\Mm$ be an extension of $\ker\alpha$ (which is a maximal ideal) to $A$ (as a two-sided ideal). Observe that this extension need not be unique. 

The quotient $A/\Mm$ splits into a finite direct product of simple algebras $\bar\xi_i:=A/\Mm_i$, where $\Mm:=\Mm_1\Mm_2\cdots\Mm_s$. Hence $k_i:=\cent(\bar\xi_i)$ are fields  and so the projection $\xibf:A\to A/\Mm$ defines an \'etale $A/\Mm$-rational point with underlying \'etale point $\{\bar\xi_i\}$ and residue ring $E:=k_1\times k_2\times\cdots\times k_s$. In fact, for each $i$, the diagram
\begin{equation}\label{eq:diag_restrict_rational}
\begin{split}
\xymatrix{&\ker\alpha\ar@{^(->}[dl]\ar@/^3pc/[ddd]^0\\
B\ar[r]\ar[dd]&A\ar@{->>}[d]\\
&\bar\xi_i=A/\Mm_i\\
L\ar[r]&k_i=\cent(\bar\xi_i)\ar@{^(->}[u]}
\end{split}
\end{equation}shows that $k_i=L$ for all $i$. As a consequence, $E=L^s$.

Conversely, any $S$-rational point $\xibf: A\to S$ obviously restricts to an $S$-rational point on $Y=\Spec(B)$. The restriction of $\ker\xibf$ to $B$ defines a finite set of closed points on $Y$ with residue fields $L_i$. The centre of each simple component of $A/\ker\xibf$ is a field $k_i$ and $L_i\subseteq k_i$. In fact, a reasoning similar to the above diagram (\ref{eq:diag_restrict_rational}) shows that $k_i=L_i$.

\begin{thm}[M\"uller's theorem]
Let $A$ be an affine PI-algebra over a field $k$ and let $M$ and $N$ be simple finite-dimensional $A$-modules. Then 
$$\Ext^1_A(M,N)\neq \emptyset\quad \iff\quad \Ann(M)\cap \cent(A)=\Ann(N)\cap \cent(A).$$(The annihilators are \emph{left} ideals.) 

Note that $\Ann(M)$ and $\Ann(N)$ are maximal ideals since $M$ and $N$ are simple. Hence the intersection is also a maximal ideal. Hence, putting $\mm:=\Ann(M)\cap \cent(A)=\Ann(N)\cap \cent(A)$, we can rephrase the equivalence as
$$\Ext^1_A(M,N)\neq \emptyset\quad \iff\quad \mm\in\ram(A).$$
\end{thm}
\begin{proof}
	This is a reformulation of M\"uller's theorem as stated in \cite[Theorem III.9.2]{BrownGoodearl} using \cite[Lemma I.16.2]{BrownGoodearl}. 
\end{proof}

If $M$ is a simple left $A$-module then $M\simeq A/\Ann(M)$ as left $A$-modules. In the case where $M=A/\mm$ for some maximal ideal $\mm$, we see that, as left ideals, $\mm=\Ann(M)=\Ann(A/\mm)$. A module $M$ is simple if and only if $M\simeq A/\Ann(M)$, from which it follows that $A/\mm$ is simple as left $A$-module.   

In fact, for $A$ a PI-algebra over a Jacobson ring $B$, $B\subseteq \cent(A)$, Theorem 13.10.4 in \cite{McConnellRobson} implies that if $M$ is a simple $A$-module, then $\Ann(M)$ is a maximal ideal and $M$ is finite-dimensional over $B/(\Ann(M)\cap B)$. In addition, $A$ satisfies the non-commutative Nullstellensatz. 

Let $\xibf\in\Xscr_A(S)$ be an \'etale rational point on $\Xscr$, with $S$ an artinian $k$-algebra. Then
$$A/\ker\xibf \simeq \prod_{i=1}^s A/\Mm_i=\prod_{i=1}^s \bar\xi_i,$$and where $\xi_i: A\to A/\Mm_i$ are the associated closed points (with underlying points $\bar\xi_i=A/\Mm_i$; recall that this are \emph{algebras}).  By the previous paragraph the $\xi_i$ are simple. Therefore, M\"uller's theorem can be rephrased in terms of rational points as the equivalence
$$\Ext^1_A(A/\Mm_i,A/\Mm_j)\neq \emptyset\quad \iff\quad \ker\xi_i\cap \cent(A)=\ker \xi_j\cap \cent(A).$$We will sometimes write $\Ext^1_A(\xi_i,\xi_j)$ for $\Ext^1_A(A/\Mm_i,A/\Mm_j)$. Observe that 
$$\big(\Ext^1_A(\xi_i, \xi_j)_{ij}\big)=\mathbf{T}_{\rho}=\mathbf{T}_{\{\xi_i\}}.$$ Despite the theorem, there is nothing saying that $\{\bar\xi_i\}\subset \ram(A)$ for all $i$ in an \'etale point. This certainly depends on $\xibf$. 

Conversely, for all closed points $\mm$ in $\Spec(\cent(A))$, the fibre $\Psi^{-1}(\mm)$ (recall that $\Psi$ is algebraically the inclusion $\cent(A)\hookrightarrow A$) is an \'etale rational point. Indeed, let $\mm=\Mm_1\Mm_2\cdots\Mm_s$ be the decomposition of $\mm$ in $A$. Then 
$$A/\mm=A/\Mm_1\times A/\Mm_2\times\cdots\times A/\Mm_2$$ and so $\xibf: A\to A/\mm$ is an \'etale $A/\mm$-rational point with underlying points $\bar\xi_i=A/\Mm_i$. Diagram (\ref{eq:diag_restrict_rational}) once again shows that $\cent(\bar\xi_i)=k(\mm)$ for all $i$. 

\begin{prop}\label{prop:ExtAcent}Let $\mm\in\azu(A)$. Then
$$\Ext^1_A(A/\mm, A/\mm)\simeq \Ext^1_{\cent(A)}(k(\mm), k(\mm)).$$In other words, the central simple algebra $A/\mm$ have the same deformation theory as $\cent(A)$ over $\azu(A)$. Since $\Mm=A\mm A$ is maximal in $A$, we have $A/\mm=A/\Mm$.
\end{prop}
\begin{proof}
	We will use the following change of base theorem for $\Ext^1$. Let $R\to S$ be a ring morphism, $M_R$ a left $R$-module and $M_S$ a left $S$-module. Then 
	$$\Ext^1_S(M_R\otimes_R S, M_S)\simeq \Ext^1_R(M_R, M_S).$$Take $R=\cent(A)$, $S=A$, $M_R=\cent(A)/\mm$ and $M_S=A/\mm$. Then
	$$\xymatrix{\Ext^1_A(\cent(A)/\mm\otimes_{\cent(A)}A, A/\mm)\ar@{=}[d]\ar[r]^\simeq & \Ext^1_{\cent(A)}(\cent(A)/\mm, A/\mm)\\
	\Ext^1_A(A/\mm, A/\mm)}$$Representing $\Ext$ in terms of Hochschild cohomology we have
	$$\Ext^1_A(M, N)\simeq \Der_k(A, \Hom_k(M, N))/\mathrm{Ad},$$ where $\mathrm{Ad}$ is the group of inner derivations. Thus there is a surjection
	$$\Ext^1_{\cent(A)}(\cent(A)/\mm, A/\mm)\twoheadrightarrow \Der_k(\cent(A), \Hom_k(\cent(A)/\mm, A/\mm))$$ with kernel $\mathrm{Ad}$. 
	
	Now, any $\phi:\cent(A)/\mm\to A/\mm$ gives a morphism $\cent(A)/\mm\to \cent(A)/\mm$ by restriction. Conversely, any $\psi: \cent(A)/\mm\to \cent(A)/\mm$ certainly gives a morphism $\cent(A)/\mm\to A/\mm$ by composing with the injection. Hence, $$\Hom_k(\cent(A)/\mm, A/\mm)=\Hom_k(\cent(A)/\mm, \cent(A)/\mm)$$ and so 
\begin{align*}
\Der_k(\cent(A), \Hom_k(\cent(A)/\mm, A/\mm))&=\Der_k(\cent(A), \Hom_k(\cent(A)/\mm, \cent(A)/\mm))\\
&=\Der_k(\cent(A), \Hom_k(k(\mm), k(\mm))),
\end{align*}which proves the proposition. 
	\end{proof}
\begin{thm}Let $\xibf: A\to S$ be an \'etale $S$-rational point on $\Xscr_A$.  
\begin{itemize}
	\item[(a)] Suppose $\ker\xibf\cap\cent(A)\subset\mathtt{smooth}(A)\subseteq\azu(A)$, where $\mathtt{smooth}(A)$ is the smooth locus of $\cent(A)$.
	Then the deformations of $\{\xi_i\}$ are unobstructed and the versal family (``formal $S$-rational point'') is
	$$\hat\xibf=(\hat\xi_i): A\to \bigoplus_{i=1}^s \End_k(\xi_i)\otimes_k k\langle\!\langle t_1, t_2, \dots, t_{n}\rangle\!\rangle,$$ where $$n= \dim_k(\Ext^1_A(\xi_i, \xi_i))=\dim_k(\Ext^1_{\cent(A)}(k(\mm), k(\mm)))=\dim_k(T_\mm Y),$$where $T_\mm Y$ denotes the tangent space at $\mm\in Y=\Spec(\cent(A))$. 
	
	Observe that, since $\mm\in\mathtt{smooth}(A)$, the dimension of the $\Ext$-spaces is independent on the $\xi_i$. 
	\item[(b)] If $A$ is prime, of finite global dimension (for instance, if $A$ is Auslander-regular) and finite as a module over its centre, then
	$$\mathtt{sing}(A)\cap\azu(A)=\emptyset.$$Hence, the claims of (a) applies to all of $\azu(A)$ (and any singularities must lie in $\ram(A)$). 
	\item[(c)] Assume $\ker\xibf\cap\cent(A)\subset \ram(A)$. Then the versal family (``formal $S$-rational point'') is
		$$\hat\xibf=(\hat\xi_i): A\to \Big(\Hom_k(\xi_i, \xi_j)\otimes_k \hat H_{ij}\Big), $$ where 
		$$\hat H_{ii}\simeq k\langle\!\langle t_1, t_2, \dots, t_{n_i}\rangle\!\rangle\big/ (f_1, f_2, \dots, f_{m_i}),$$ and
		$$n_i = \dim_k(\Ext^1_A(\xi_i, \xi_i))\quad\text{and}\quad m_i=\dim_k(\Ext^2_A(\xi_i, \xi_i)).$$The elements off the diagonal are more complicated to express in general, and are not algebras but ideals. 
		
		Note that, in the smooth part of the ramification locus, the deformations are unobstructed (i.e., $f_1=f_2=\cdots=f_{m_i}=0$ for all $i$). 
\end{itemize}The case where $\xi$ intersects both $\azu(A)$ and $\ram(A)$ clearly splits into two disjoint cases. 
\end{thm}

\begin{proof}The claim concerning the unobstructedness in (a) follows from proposition \ref{prop:ExtAcent} and the fact that smooth points deforms without obstruction. The versal family is given a direct consequence of the definition of unobstructed deformations of representations as given in section \ref{sec:deformation}. Proposition \ref{prop:ExtAcent} also implies the claim concerning the dimensions. Part (b) follows from \cite[Lemma III.1.8]{BrownGoodearl}. The claims in (c) is also follows from the discussion in section \ref{sec:deformation}.
\end{proof}
The above gives a complete description of the $\OO$-rings in the case when $A$ is PI-algebra.

\section{Non-commutative Diophantine Geometry}
\subsection{Height functions}\label{sec:cent_height}A height function on an algebraic variety $X_{/k}$ over a number field $k$ (or function field for that matter) is a function 
$$H_K:\,\, X_{/k}(K)\longrightarrow \R, \qquad k\subseteq K,$$i.e., a function on (the coordinates of) $K$-rational points with values in $\R$.  We will follow \cite[Chapter B]{HindrySilverman}, for which we refer for more details. Put
$$\R(X):=\Big\{H: X(k^{\mathrm{alg}})\to \R\Big\}$$ and call elements in this set \emph{height functions}. If we were to take a more serious and in-depth look at the subject, we should consider $\R(X)$ modulo bounded functions. However the above is more than sufficient for our purpose here. 

The fundamental example is the following. Let $X$ be the $n$-dimensional projective space $\P^n$ and let $\Sigma_k$ be the set of valuations of $k$:
$$\Sigma_k = \Sigma_k^\mathrm{f}\cup \Sigma_k^\infty,$$ where $\Sigma_k^\mathrm{f}$ is the set of non-archimedean (finite) valuations and $\Sigma_
k^\infty$, the set of archimedean (infinite) valuations. We denote the, to $v\in\Sigma_k$ associated normalized absolute value, $|\!|\cdot|\!|_v$.

Let $\boldsymbol{p}=(p_0: p_1: \cdots: p_m)\in\P^m(k)$ and define the \emph{(Weil) height} of $\boldsymbol{p}$ to be
$$H_k(\boldsymbol{p}):=\prod_{v\in\Sigma_k}\max\Big\{ |\!|p_0|\!|_v, |\!|p_1|\!|_v, \dots, |\!|p_m|\!|_v\Big\},$$and the \emph{logarithmic height} as
$$h_k(\boldsymbol{p}):=\log H_k(\boldsymbol{p})=-\sum_{v\in\Sigma_k}n_v\min\Big\{ v(p_0), v(p_1), \dots, v(p_m)\Big\},$$where $n_v= [k_v/\Q_p]$ and $v\mid p$. 

For $k\subseteq K$ a finite extension, \cite[Lemma B.2.1(c)]{HindrySilverman} gives 
\begin{equation}\label{eq:heightsequence}
H_K(\boldsymbol{p})=H_k(\boldsymbol{p})^{[K/k]}.
\end{equation}The following is therefore a natural definition: we define a \emph{height sequence} to be a sequence $\{H_k\mid \Q\subseteq k\}$, parametrized by the field extensions of $\Q$, with the different $H_k$ coherent in the sense that (\ref{eq:heightsequence}) holds if $k\subseteq K$. 

We also define the \emph{absolute height} of $\boldsymbol{p}$ to be
$$H(\boldsymbol{p}):=H_K(\boldsymbol{p})^{\frac{1}{[K/\Q]}}$$ where $K$ is any field such that $\boldsymbol{p}\in\P^m(K)$. This definition is independent on $K$.  Similarly we define the \emph{absolute logarithmic height} to be
$$h(\boldsymbol{p}):=\log H(\boldsymbol{p})=\frac{1}{[K/\Q]}h_K(\boldsymbol{p}).$$

\subsubsection{Height functions on $\cent(A)$}
Now let $\P(\cent(A)):=\P(\Spec(\cent(A)))$ be the projective closure of $\Spec(\cent(A))$:
$$\mathrm{Proj}:\,\, \Spec(\cent(A))\hookrightarrow \P(\cent(A)).$$ As $\P(\cent(A))$ is a projective scheme there is a closed embedding 
$$\alpha:\,\, \P(\cent(A))\hookrightarrow\P^m$$ for some $m$. Put $\varphi:=\alpha\circ\mathrm{Proj}$ and let $\boldsymbol{p}\in\Spec(\cent(A))(K)$ be a $K$-rational point. We then define the \emph{height of $\boldsymbol{p}$ relative to $\varphi$} to be the function
$$H_\varphi(\boldsymbol{p}):= H(\varphi(\boldsymbol{p})),$$where $H$ is the absolute height function on $\P^m$. We also define the \emph{logarithmic height relative to $\varphi$} as
$$h_\varphi(\boldsymbol{p}):=h(\varphi(\boldsymbol{p})).$$

\subsubsection{The (na\"ive) central heights}

Let $\xibf\in\Xscr_A(S)$ be an \'etale $S$-rational point with $\ker\xibf=\Mm_1\Mm_2\cdots\Mm_2$. Put $\mm_i:=\Mm_i\cap\cent(A)$, viewed as points in $\Spec(\cent(A))$.

We now make the following na\"ive definition. 
\begin{dfn}
Let $\xibf\in\Xscr_A(S)$ be an \'etale $S$-rational point.  Then the \emph{central height} of $\xibf$ relative to $\varphi$ is the vector
$$H_{\varphi}^\cent(\xibf):=\Big(H\big(\varphi(\mm_1)\big), H\big(\varphi(\mm_2)\big), \dots, H\big(\varphi(\mm_s)\big)\Big),$$ with its associated logarithmic counterpart 
 $$h_{\varphi}^\cent(\xibf):=\Big(h\big(\varphi(\mm_1)\big), h\big(\varphi(\mm_2)\big), \dots, h\big(\varphi(\mm_s)\big)\Big),$$We denote by $\R^\cent(\Xscr_A)$ the set of all central heights on $\Xscr_A$. This set is clearly parametrized by the embeddings $\alpha: \P(\cent(A))\hookrightarrow \P^m$.
 \end{dfn}  
In other words, if $\Psi: \Xscr_A\to \Spec(\cent(A))$ is the morphism defined by restriction, we have $$H^\cent_\varphi = H_\varphi\circ \Psi=H\circ\varphi\circ\Psi=H\circ\alpha\circ\mathrm{Proj}\circ\Psi,$$where $\alpha$, $\mathrm{Proj}$ and $\varphi$ are defined above. In fact, since the projective closure is canonical, this construction is only dependent on $\alpha$. Hence, we write $H_\alpha$ for $H_\varphi$.

The following is a (slightly) non-commutative variant of Weil's Height Machine (see \cite[Chapter B]{HindrySilverman}).
\begin{thm}\label{thm:Weil_height_machine}Let $\AA$ be a PI-algebra with centre $\cent(\AA)$.

\begin{itemize}
	\item[(a)] We have set-theoretic maps
$$
	\xymatrix@R=5pt{\mathrm{Pic}^\text{va}_\cent(\Xscr_\AA)\ar[r]&\Pic(\gSpec(\cent(\AA)))\ar[r]^<<<<<{\circledast}&\R^\cent(\Xscr_\AA)\\
	\mathcal{L}\ar@{|->}[r]&\det\!\left((\Psi_\ast\mathcal{L})\vert_{\gSpec(\cent(\AA))}\right)\ar@{|->}[r]& H^\cent_{\alpha},}
$$
where $\alpha$ is the embedding $\alpha:\, \P(\gSpec(\cent(\AA)))\hookrightarrow \P^m$, associated with $\mathcal{L}$ via its determinant on $\gSpec(\cent(\AA))$. The map $\circledast$ is the classical Weil Height Machine. 
	\item[(b)] If $\mathcal{L}$ is not very ample, i.e., if $\det\!\left((\Psi_\ast\mathcal{L})\vert_{\gSpec(\cent(\AA))}\right)$ is not very ample, we can find  two very ample $\mathcal{L}_1$ and $\mathcal{L}_2$ on $\P(\gSpec(\cent(\AA)))$ such that
$$\det\!\left((\Psi_\ast\mathcal{L})\vert_{\gSpec(\cent(\AA))}\right)\simeq \mathcal{L}_1\otimes \mathcal{L}_2^{-1}.$$ Therefore we can define $H^\cent_\alpha:=H^\cent_{\alpha_1}-H^\cent_{\alpha_2}$, giving
$$
	\xymatrix@R=5pt{\mathrm{Pic}(\Xscr_\AA)\ar[r]&\mathrm{Pic}(\gSpec(\cent(\AA)))\ar[r]^<<<<<{\circledast}&\R^\cent(\Xscr_\AA)\\
	\mathcal{L}\ar@{|->}[r]&\det\!\left((\phi_\ast\mathcal{L})\vert_{\gSpec(\cent(\AA))}\right)\ar@{|->}[r]& H^\cent_\alpha.}
$$
\end{itemize}

\end{thm}
It is important to observe that  $(\mathcal{L}_1, \alpha_1)$ and $(\mathcal{L}_2, \alpha_2)$ are not unique. However, the height functions $H^\cent_\alpha$ coming from different choices of  $(\mathcal{L}_1, \alpha_1)$ and $(\mathcal{L}_2, \alpha_2)$, differ ``only'' up to bounded functions (see for example \cite[Theorems B.3.2 and B.3.6]{HindrySilverman}).  
\begin{proof}Everything, except the map $\circledast$, follow directly from construction. For the construction of $\circledast$, see \cite[Theorems B.3.2 and B.3.6]{HindrySilverman}. That any invertible sheaf on a projective scheme can be written as (a multiplicative) difference of two very ample ones is well-known, but let's spell out an argument nevertheless. The twists $\OO_X(n)$ are very ample for all $n$ and if $\mathcal{L}$ is an invertible sheaf, $\mathcal{L}\otimes \OO_X(n)$ is very ample for $n$ sufficiently large by \cite[Theorem II.5.17 and Exercise II.7.5(d)]{Hartshorne}. Hence $\mathcal{L}\simeq \mathcal{L}\otimes \OO_X(n)\otimes\OO_X(n)^{-1}$
\end{proof}
\begin{remark}It seems interesting to consider \emph{ramification heights}, i.e., heights associated with the ramification locus. 
\end{remark}

\subsubsection{Representation heights}
The above was perhaps the most naive and obvious notion of height possible for PI-algebras. We will now construct a more ``non-commutative version'' that works for all finitely generated algebras. For simplicity, we work with affine algebras. 

Let $\xibf\in\Xscr_A(S)$ be an $S$-rational point. We will write out the construction for non-\'etale points. The \'etale case will be obvious. 

Put $\Mm:=\ker\xibf$, $\mm:=\Mm\cap\cent(A)$ and $K:=\cent(A/\Mm)$. Note that $A/\Mm$ is a finite-dimensional $K$-vector space as it is central simple over its centre (which in turn is a finite extension of $k$). The point $\xibf$ thus defines a representation $\xi:A\to \End_K(A/\Mm)$ (via the projection $A\twoheadrightarrow A/\Mm$). 

Let $\{x_1, x_2, \dots, x_n\}$ be a set of generators for $A$. Since $A/\Mm$ is simple as $K$-algebra, there is, by Wedderburn's theorem, a unique division algebra $D$ with $\cent(D)=K$, such that $A/\Mm\simeq \mathrm{Mat}_m(D)$, where $m$ is also uniquely determined by $A/\Mm$. Hence we can view the $\xi(x_i)$ as matrices with entries in $D$.  These matrices are the \emph{coordinates} of $\xibf$.

Let $v$ be a henselian $\R$-valued valuation on $K$ and $K_v$ the completion of $K$ with respect to $v$. For any finite extension $K_v \subset L$ the $v$ extends uniquely to $L$. Therefore $v$ extends uniquely to $D\otimes_K K_v$ (see \cite[Thm. 1.4]{TignolWadsworth}) and we have a morphism
$$\beta: A/\Mm \xrightarrow{\simeq}\mathrm{Mat}_m(D)\to \mathrm{Mat}_m(D\otimes_KK_v).$$ Denote by $w_D$ the to $D_v:=D\otimes_KK_v$ \emph{uniquely} extended valuation of $v$ and put $M_i:=\beta(\xi(x_i))$. Applying the $D$-valuation $w_D$ to all entries in $M_i$ we get matrices $w_D(M_i)\in\mathrm{Mat}_m(\R)$.

\begin{dfn}
Put $d_i:=\det(w_D(M_i))$. Then we define the \emph{logarithmic height} of $\xibf$ as.
$$h_K^\mathtt{rep}(\xibf):= -\sum_{v\in\Sigma_K}n_v\min\{d_1, d_2, \dots, d_{m_i}\}, \quad n_v:=[D\otimes_K K_v/k].$$The corresponding absolute height is defined as $H_K^\mathtt{rep}(\xibf):=\exp(h_K^\mathtt{rep}(\xibf))$. 
\end{dfn}

\subsubsection{Non-commutative heights}

Let $\xibf\in\Xscr_A(S)$ and decompose the kernel as $\ker\xibf = \prod_i^s\Mm_i$. Put, in addition, $\mm_i:=\Mm_i\cap\cent(A)$. 
$$\mathsf{P}:=\Psi^{-1}(\Psi(\xibf))=\big\{\Psi^{-1}(\mm_i)\,\,\big\vert\,\,1\leq i\leq s\big\}=\big\{\pp_1, \pp_2,\dots,\pp_r\big\}, \quad r\geq s.$$
Notice that this set can include more points than the points in the decomposition of $\ker\xibf$, depending on whether some of the $\mm_i\in\ram(A)$ or not.

The set $\mathsf{P}$ defines a new \'etale point:
$$\boldsymbol{\varrho}=(\varrho_i): A\to \prod_{i=1}^rA/\pp_i$$ with and underlying points $\bar\varrho_i$. 
Put $(\TE_\mathsf{P})_{ij}:=\Ext^1_A(\varrho_i, \varrho_j)$ (cf. section \ref{sec:naive_nc_simp}). 

The augmented tangent space graph $\TG_\mathsf{P}=\TG_\mathsf{P}^1$ then measures the noncommutativity of the point $\xibf$. Put $e_{ij}:=\dim_k((\TE_\mathsf{P})_{ij})$.

The adjacency matrix of $\TG_\mathsf{P}$ is  
$$M_\mathsf{P}:=\begin{pmatrix}
	e_{11} & e_{12} & \cdots & e_{1r}\\
	e_{21}& e_{22} & \cdots & e_{2r}\\
	\vdots & \vdots & \ddots& \vdots\\
	e_{r1}&e_{r2}&\cdots & e_{rr}
\end{pmatrix},$$encoding the $\TG_\mathsf{P}$ in matrix form. Two graphs are isomorphic if their adjacency matrices are conjugate (similar). Hence, up to conjugacy (similarity), the ordering of the points $\xibf_i$ (and thus the $\pp_i$) is irrelevant.  As a consequence, the set $\Lambda$ of eigenvalues is an invariant of $\TG_\mathsf{P}$ and $\TE_\mathsf{P}$.

Recall that $K=\cent(A/\Mm)$. We then make the following definition.
\begin{dfn}The \emph{non-commutative height} of $\xibf$ is 
$$H_K^\mathtt{nc}(\xibf):=\prod_{\sigma: K\hookrightarrow K^\mathrm{al}}\max\Big\{\vert\!\vert\sigma(\lambda_1)\vert\!\vert, \vert\!\vert\sigma(\lambda_2)\vert\!\vert, \dots, \vert\!\vert\sigma(\lambda_t)\vert\!\vert\Big\}.$$The \emph{total height} of $\xibf$ is the vector
$$H_{K, \alpha}^\mathtt{tot}(\xibf):=\Big(H_\alpha^\cent(\xibf), H_K^\mathtt{rep}(\xibf),H_K^\mathtt{nc}(\xibf)\Big)\in\R^2\times\C. $$Observe that this dependent on the embedding $\alpha:\mathbb{P}(\cent(A))\hookrightarrow \mathbb{P}^n$. According to the Height Machine $H_\alpha^\cent(\xibf)$ can be given in terms of a very ample sheaf on the central scheme and a choice of $\alpha$. 
\end{dfn}


\section{Arithmetic geometry of PI-algebras}\label{sec:arithPI-alg}
\subsection{Compactifying the base}\label{sec:arakelov}

Let $k$ be a number field and $\oo$ an order in $k$ (i.e., $\oo$ need not be integrally closed in $K$; we will later assume this though). Assume that $X$ is of finite type over $\oo$.

\subsubsection{Compactification of orders}
Most of what will follow in this subsection can be found in \cite[Chapter 3]{Neukirch}, although we will frame it in terms of pseudo-divisors. 

Let $\oo$ be an order in a number field $k$ and let $\Sigma$ be the set of infinite primes, i.e., the set of embeddings $k\hookrightarrow \C$. We compactify $\oo$ as 
$$\ooo:=\oo\times \Sigma.$$Put $Y:=\Spec(\ooo)$. Hence $\Sigma = Y(\C)$. We will sometimes use the notation $Y^\mathrm{f}:=\Spec(\oo)$ and $Y^\infty:=\Sigma$.

A finitely generated $\oo$-module $M$ extends to a module over $\ooo$ by extending $M$ to 
$$M_\C=M\otimes_\Z \C = \bigoplus_{\sigma\in\Sigma}M_\sigma=\bigoplus_{\sigma\in\Sigma}M\otimes_{\oo, \sigma}\C, \qquad M_\sigma = M\otimes_{\oo, \sigma}\C,$$ where $M\otimes_{\oo, \sigma}\C$ of course means that we view $\C$ as an $\oo$-module via $\sigma:\oo\hookrightarrow \C$. We put
$$\widehat{M}:=M\times M_\C.$$ 
\begin{dfn}Let $Y=\Spec(\ooo)$. Then an \emph{Arakelov--Cartier divisor} on $Y$ is a pair of pseudo-divisors
$$\big(\widehat\Lscr, \widehat{Z}, \hat{s}\big):=\left(\Lscr^\mathrm{f}, \Lscr^\infty\right):=\Big((\Lscr, Z, s), (\Lscr_\C, Z_\C, s_\C)\Big),$$where $Z_\C\subseteq \Sigma$ and $s_\C$ a function $\Lscr_\C\to \C^\times$, non-vanishing on $\Sigma\setminus Z_\C$. If $\widehat{Z}$ and $\hat{s}$ are given or irrelevant we simply write $\widehat{\Lscr}$ for $\big(\widehat{\Lscr}, \widehat{Z}, \hat{s}\big)$.
\end{dfn}
Notice that $\Lscr_\C$ need not be the base change of $\Lscr$ to $\C$. 

Let $D$ be a Cartier divisor on $Y^\mathrm{f}$. Recall that a Cartier divisor on $Y^\mathrm{f}$ can be identified with the invertible ideals of $\oo$ (i.e., the finitely generated $\oo$-modules $I\subset k$ such that there is another finitely generated  $I^{-1}\subset k$ with $I\otimes I^{-1}=\oo$). 

Then $D$ determines an Arakelov--Cartier divisor $\left(\mathcal{L}^\mathrm{f}, \mathcal{L}^\infty\right)$ with
$$\mathcal{L}^\mathrm{f}= \big(\OO(D), |D|, s\big),\quad\text{and} \quad\mathcal{L}^\infty = \big(\OO(D)_\C, \Sigma, 1_\Sigma\big).$$To spell it out explicitly, let $$D=\Big\{(U_i, f_i)\mid 1\leq i\leq m, \,\, f_i\in k\Big\}$$ be a Cartier divisor on $Y$. The support of $D$ is 
$$|D|=\Big\{\pp\in\Spec(\oo)\mid f_i\notin \oo_\pp^\times \text{ for some $i$}\Big\}.$$Then 
$$\OO(D)=\oo\cdot f_1^{-1}+\oo\cdot f_2^{-1}+\cdots+\oo\cdot f_m^{-1}, \qquad s=\prod_{i=1}^m f_i$$(we write $s$ multiplicatively) and 
$$\OO(D)_\C=\OO(D)\otimes_\Z \C=\prod_{\sigma\in\Sigma} \OO(D)\otimes_{\oo, \sigma}\C, \quad\text{with}\quad 1_\Sigma(\sigma)=1\in\C^\times, \,\, \sigma\in\Sigma.$$Obviously, by $\Sigma$ we mean all embeddings of $k\hookrightarrow \C$ as before.

\subsubsection{Compactification of algebras over arithmetic schemes}
We compactify $X_{/\oo}$ to a scheme $\widehat{X}_{/\ooo}$ by adding the complex points of $X$:
$$\widehat{X}:=\widehat{X}_{/\ooo}:=X_{/\oo}\times X^\infty, \quad\text{where} \quad X^\infty:=\prod_{\sigma\in\Sigma} X\times_{\oo, \sigma}\Spec(\C).$$We call $X:=X_{/\oo}$ the \emph{finite} part (we include the generic fibre in the finite part) and $X^\infty$ the \emph{analytic} (or \emph{infinite}) part of $\widehat{X}_{/\ooo}$.

Let $\AA$ be a  finitely generated algebra over $X$, with structure morphism $f:\OO_X\to\AA$. We extend this to an analytic algebra as follows. Fix an embedding $\sigma: \oo\hookrightarrow \C$. Put
$$\OO_X\otimes_{\oo, \sigma}\C\xrightarrow{f\otimes_\sigma\C} \sigma^\ast\AA:=\AA\otimes_{\oo, \sigma}\C$$and
$$\prod_{\sigma\in\Sigma}  \Big(\OO_X\otimes_{\oo, \sigma}\C\xrightarrow{f_\sigma}\sigma^\ast\AA\Big), \quad f_\sigma:= f\otimes_\sigma\C.$$We will normally work with one $\sigma$ at a time for simplicity of notation. Finally we put
$$\widehat{\AA}:=\AA\times\prod_{\sigma\in\Sigma}\sigma^\ast\AA.$$ We call $\widehat{\AA}$ the \emph{compactification} of $\AA$ over $\widehat{X}$. 

Let $\MM$ be a left $\AA$-module, finite over $\OO_X$. We extend $\MM$ to an $\widehat{\AA}$-module by 
\begin{align*}
\widehat{\MM}:=\MM\times\prod_{\sigma\in\Sigma}\sigma^\ast\MM,\quad \text{with}\quad\sigma^\ast\MM:=\MM\otimes_{\oo, \sigma}\C
\end{align*}
Hence, $\sigma^\ast\MM$ are sheaves over $X^\infty$ with an action by $\AA$ and where $\oo$ acts via $\sigma$.

The category of all $\widehat{\AA}$-modules is denoted $\Mod(\widehat{\AA})$ and is the product category
$$\Mod(\widehat{\AA}):=\Mod(\AA)\times\prod_{\sigma\in\Sigma}\Mod(\sigma^\ast\AA).$$This meaning of this notation is hopefully clear.  

\begin{dfn}We define $\widehat{\Xscr}_\AA$ as
\begin{align*}\widehat{\Xscr}_\AA&:=(\Xscr_\AA^\mathrm{f},\Xscr_\AA^\infty):=\Mod(\AA)\times\prod_{\sigma\in\Sigma}\Mod(\sigma^\ast\AA)\\
&:=\underline{\Mod}(\widehat{\AA})=\big(\Mod(\widehat{\AA}), T_\mathrm{ZJ}, \widehat{\boldsymbol{\OO}}\big).
\end{align*}Clearly, $\widehat{\boldsymbol{\OO}}$ decomposes as 
$$\widehat{\boldsymbol{\OO}}:=\boldsymbol{\OO}^\mathrm{f}\times\boldsymbol{\OO}^\infty:=\boldsymbol{\OO}\times\prod_{\sigma\in\Sigma}\sigma^\ast\boldsymbol{\OO}$$ and the same applies to the topology $T_\mathrm{ZJ}$.
\end{dfn}Do not confuse $\widehat{\boldsymbol{\OO}}$ (the object defined above) and $\hat{\boldsymbol{\OO}}$ (the formal object coming from deformation theory). The notation here is not optimal but I think it will not cause too much headache for the reader. 
 
Observe that there are properties of $\AA$ that can be shown to hold over the analytic part, that does not necessarily hold over the finite part due to the fact that the base field is algebraically closed. 

\subsection{Divisors}

For the sake of simplicity we express the next definition in terms of affine algebras. Hence $X=\Spec(B)$ and $\AA = A$ with $B$ a finitely generated $\oo$-algebra with generators $\{x_1, x_2, \dots, x_s\}$, where $\oo$ is an order in a number field $k$. We also assume that $A$ is a prime ring. This implies that $\cent(A)$ is a domain. 

Since $A$ is prime, Posner's theorem implies that there is a central simple algebra $\Frac(A)$ in which $A$ is a maximal order. In fact, 
$$\Frac(A)=A\otimes_{\cent(A)}\cent(A)_{(0)}=A\otimes_{\cent(A)}k(\cent(A)),$$where $\cent(A)_{(0)}$ denotes localisation at the generic point, and $k(\cent(A))$ the field of quotients (and these two are the same). 

Choose generators $\{\eb_1, \eb_2, \dots, \eb_r\}$ of $A$ over $\cent(A)$, where $r=\mathrm{rk}_\cent(A)$. Let $\Lscr_\cent$ be an invertible subsheaf of $k(\cent(A))$ and choose $s$ families of global sections $$\Big\{\beta_{i;k}\,\,\big\vert\,\, 1\leq i\leq s\Big\}_{k=1}^r$$of $\Lscr_\cent$. Choose a covering $\{D_j\}$ of $\cent(A)$.

We introduce the following $\cent(A)$-action, using $\Lscr_\cent$, on $A$:
$$(x_i\cdot \eb_k)(D_j):=\beta_{i; k}(D_j)\eb_k,\quad 1\leq i\leq s,\,\,\, 1\leq k\leq r.$$ Define an element 
$$\alphabf(D_j):=\sum_{i=1}^r \alpha_i(D_j)\eb_i, $$ with the $\alpha_i(D_j)$ sections of $\Lscr_\cent$ over $D_j$. Put $\Lscr:=A\cdot\alphabf\cdot A$. This defines an invertible $A$-module over $\cent(A)$ (in the sense of section \ref{sec:sheaves}) and
$$\Lscr(D_j)=(A\cdot\alphabf\cdot A)(D_j)=\bigoplus_{i=1}^r	A(D_j)\cdot \alpha_i(D_j)\eb_i\cdot A(D_j),$$ where $$A(D_j):=A\otimes_{\cent(A)}\cent(A)_{g_j}\quad\text{and}\quad D_j=\Spec(\cent(A)_{g_j}).$$ We say that $\Lscr\in\Pic_{\cent(A)}(A)$ as defined above is a \emph{line sheaf} (resurrecting S. Lang's terminology) over $A$.

The \emph{support}, $|\Lscr|$, of $\Lscr$ is 
$$|\Lscr| := \Xscr_{A/\mathcal{J}}, $$with $\mathcal{J}$ the two-sided ideal sheaf $$\mathcal{J}(D_j):=\Big\{A(D_j)\cdot\beta_{i;k}^{-1}(D_j)\cdot A(D_j)\,\,\big\vert\,\, 1\leq i\leq s,\,\,\, 1\leq k\leq r\Big\}.$$

Observe that the above constructions are essentially obvious but it is helpful to spell them out nevertheless. 

\begin{dfn}Let the data above be given.
\begin{itemize}
	\item[(i)] A \emph{Cartier divisor} on $\Xscr_A$ is a pair $(\Lscr, s_\Lscr)$, where $\Lscr$ is a line sheaf on $\Xscr_A$ and $s_\Lscr$ a global section of $\Lscr$, as constructed above.
	\item[(ii)]  An \emph{Arakelov--Cartier} divisor on $\widehat{\Xscr}_A$ is a pair $\widehat{\Lscr}:=\big(\Lscr^{\mathrm{f}}, \Lscr^\infty\big)$, where $\Lscr^\mathrm{f}$, is a Cartier divisor on the finite part $\Xscr_A^{\mathrm{f}}$ and $\Lscr^\infty$ a Cartier divisor on the analytic part $\Xscr_A^\infty$. 
	\item[(iii)] The divisor $(\widehat{\Lscr}, s_{\widehat{\Lscr}})$, where $s_{\widehat{\Lscr}}:=(s_{\Lscr^\mathrm{f}}, s_{\Lscr^\infty})$,  is \emph{effective} if all $\beta_{i;k}^{-1}\in \cent(A)$.
\end{itemize}
\end{dfn}

Let $X\xrightarrow{f} \Spec(\oo)$ be an \emph{arithmetic surface} (i.e., $X$ has relative dimension one over $\oo$), then a \emph{vertical (or fibral) divisor} $\dd$ is a divisor included in a fibre $X\otimes_{\oo}k(\qq)$, for $\qq\in\Spec(\oo)$. Equivalently, $\dd$ is vertical if $f(\dd)=\{\qq\}$. In addition, a divisor $\dd\subset X$ such that $f(\dd)=\Spec(\oo)$ is called a \emph{horizontal divisor}. Equivalently, $\dd$ is horizontal if it is the Zariski closure of a closed point on the generic fibre of $X$.  

In terms of affine algebras a prime divisor is a codimension-one prime $\pp\subset B$, i.e., a prime such that $\dim(B/\pp)=1$. Hence, $\pp$ is fibral if $\pp=f^a(\qq)$, for some $\qq\in\Spec(\oo)$; $\pp$ is horizontal if $\pp=f^a(\oo)$. Here $f^a$ is the to $f$ associated algebraic map.

Recall that $\dim(A)$ denotes the classical Krull dimension, i.e., the supremum of all chains of 2-sided prime ideals in $A$.  Hence, saying that a prime $\pp$ has codimension $n$ means that $\dim(A/\pp)=n$. 
 
 Let $\pp\subset A$ be a 2-sided codimension-one prime in $A$, and put $\pp_\cent:=\pp\cap\cent(A)$ with residue class field $k(\pp):=k(\pp_\cent)$. Let $\varrho$ be the representation $\varrho: A\to \End_{k(\pp)}(A/\pp)$. The versal family of $\varrho$ is
$$\hat\varrho: A\longrightarrow \End_{k(\pp)}(A/\pp)\otimes_{k(\pp)} \hat H_{A/\pp},$$ where $\hat H_{A/\pp}$ is the pro-representing hull of $\varrho$. Recall that this is a local (non-commutative) ring. Let $\mm_{\hat H}$ be the maximal ideal. We define an order function associated with $\pp$ as
$$\mathrm{ord}_\pp(f):=\min\Big\{n\in\Z_{\geq 0}\,\,\big\vert\,\, \hat\rho\vert_{\hat H}(f)\in\mm_{\hat H}^n\Big\}, \quad f\in A,$$extended to $f/g\in\Frac(A)$ as usual by
$$\mathrm{ord}_\pp(f/g):=\mathrm{ord}_\pp(f)-\mathrm{ord}_\pp(g).$$
\begin{remark}When $A$ is commutative, the above construction reduces to the classical commutative situation. For instance, when $A$ is commutative, we have an isomorphism 
$$\hat \OO_{\Spec(A), \pp}\simeq \hat H_{A/\pp}$$ and so $\mathrm{ord}_\pp$-function reduces to the commutative order function.
\end{remark}
\begin{dfn}Let $X=\Spec(A)\to \Spec(\oo)$ be an affine scheme over $\oo$ and let $\Xscr_A^{(1)}$ denote the set of 2-sided codimension-one  primes in $A$. 
\begin{itemize}
	\item[(a)] A \emph{Weil divisor} on $\Xscr_A$ is a formal sum 
	$$\dd=\sum_{\pd\in\Xscr_A^{(1)}} n_\pd\cdot \pd, \quad n_\pd\in\Z,$$ where all but finitely many $n_\pd$ are zero. The divisor is \emph{effective} if all $n_\pd\geq 0$. The primes $\pd\in \dd$ such that $n_\pd\neq 0$ are the \emph{prime divisors} of $\dd$. Extending the definition to the analytic part in the obvious way, we can speak of \emph{Arakelov--Weil divisors}. 
	\item [(b)] If $\Spec(\cent(A))$ is an arithmetic surface, a Weil divisor $\dd$ on $\Xscr_A$ is \emph{vertical} (\emph{horizontal}) if the intersection $\dd\cap\cent(A)$ is a vertical (horizontal) Weil divisor on $\cent(A)$.
	\item [(c)]   Let $(\Lscr, s_\Lscr)$ be a Cartier divisor on $\Xscr_A$. Then the associated \emph{Weil divisor} is the formal sum
$$\mathrm{Weil}(\Lscr):=\sum_{\pd\in \Xscr_A^{(1)}}\mathrm{ord}_\pd(s_\Lscr)\pd,$$ where  denotes the set of (2-sided) codimension-one primes of $A$. 
\end{itemize}
When viewing $\pd$ as a prime ideal we write it as $\pp$, and vice versa. 
\end{dfn}

Let $\xibf: A\to S$ be an \'etale $S$-rational point such that 
$$\ker\xibf=\prod_{i=1}^n\pp_i^{n_i}, \quad n_i\geq 1,$$where at least one of the $\pp_i$ have codimension one. The underlying points are $\bar\xi_i=A/\pp_i$. Notice that, unless $\pp_i$ is maximal, $\bar\xi_i$ is an open point. Then $\xibf$ defines a Weil divisor 
$$\dd_{\xibf}=\sum_{\pp_i\in\ker\xibf}n_i\pd_i,$$ where $\pd_i$ is the divisor corresponding to $\pp_i$. Conversely, given a Weil divisor 
$$\dd=\sum_{\pd\in \Xscr_A^{(1)}}n_\pd\pd,$$ we can define an $S$-rational point:
$$\xibf_\dd: A\longrightarrow S, \quad\text{with}\quad S=\prod_{i=1}^sA/\pp_i^{n_i},$$with $\pp_i$ once again corresponding to $\pp_i$. Hence, Weil divisors are essentially \'etale rational points where the underlying points are of codimension one. Observe that the points can be non-reduced. 

We can define linear equivalence by restricting to the centre. Put
$$\dd:=\sum_{\pd\in\Xscr_A^{(1)}} d_\pd\cdot \pd, \quad\text{and}\quad \mathtt{E}:=\sum_{\pd\in\Xscr_A^{(1)}} e_\pd\cdot \pd, \qquad d_\pd, e_\pd \in\Z.$$
The restriction to the centre gives
$$\dd_\cent:=\sum_{\pd\in\Xscr_A^{(1)}} d_\pd\cdot (\pd\cap\cent(A))$$ and similarly with $\mathtt{E}$. If $\pp$ has codimension one so does $\pp\cap\cent(A)$ (follows from the going up theorem \cite[Theorem 8.14(ii)]{McConnellRobson}), so the above defines a Weil divisor on $\cent(A)$. We say that $\dd$ and $\mathtt{E}$ is \emph{linearly equivalent}, writing $\dd\sim \mathtt{E}$, if $\dd_\cent\sim \mathtt{E}_\cent$. 

The free abelian group of all divisors on $\Xscr_A$, modulo those linearly equivalent to the zero divisor, is the (first) \emph{Chow group}, $\mathrm{CH}^1(\Xscr_A)$. This extends naturally to the analytic part, allowing us to define $\mathrm{CH}^1(\overline{\Xscr}_A)$ in the obvious way. 

Using the linear equivalence of divisors, we can define the same notion for Cartier divisors. Let $(\Lscr, s_\Lscr)$ and $(\mathcal{K}, s_\mathcal{K})$ be two Cartier divisors on $\Xscr_A$. We then define $(\Lscr, s_\Lscr)$ and $(\mathcal{K}, s_\mathcal{K})$ to be linearly equivalent if the associated Weil divisors are. In this way we can define a group of Cartier divisors $\mathrm{Cart}(\Xscr_A)$, using the structure on $\mathrm{CH}^1(\Xscr_A)$. As above this extends to the analytic part and we can introduce $\mathrm{Cart}(\overline{\Xscr}_A)$.
\subsection{Intersection products}
In this section we make a rudimentary attempt at defining an intersection theory on $\Xscr_A$. A more sophisticated method involving, among other things, the infinite part (i.e., true Arakelov theory) should possibly be discussed at a later stage. 

We will write intersection products on $\Xscr_A$ as $\odot$. 

Let $\dd$ and $\ee$ be prime Weil divisors. If $\dd_\cent, \ee_\cent\subset \azu(A)$ we define the intersection product of $\dd$ and $\ee$ in $\Xscr_A$ as the \'etale rational point
$$\dd\odot\ee:=\Big(\xibf: A\longrightarrow\frac{A}{A(\dd_\cent\cap \ee_\cent) A}\Big). $$ 

 The intersection number is defined as
$$\mathsf{i}(\dd,\ee):=\sum_{\mathfrak{s}\in \dd_\cent\cap \ee_\cent}\mathrm{length}\left(\frac{A\otimes_{\cent(A)}\OO_{\cent(A), \mathfrak{s}}}{A(\dd_\cent\cap \ee_\cent)A}\right),$$ where $$\mathsf{i}_\mathfrak{s}(\dd, \ee):=\mathrm{length}\left(\frac{A\otimes_{\cent(A)}\OO_{\cent(A), \mathfrak{s}}}{A(\dd_\cent\cap \ee_\cent)A}\right)$$ is the local intersection number at $\mathfrak{s}.$ This is well-defined since $A\otimes_{\cent(A)}\OO_{\cent(A), \mathfrak{s}}$ is an Azumaya algebra of finite rank over $\OO_{\cent(A), \mathfrak{s}}$ and the set $\dd_\cent\cap \ee_\cent$ is finite. 

Let's look at the ramification locus and make the following definition.
\begin{dfn}\label{def:priminters2}
Suppose $\dd$ and $\ee$ are as above but such that $\dd_\cent\cap\ee_\cent\subset\ram(A)$.  Assume first that the intersection is one point $\pp$  and that
$$\fam{M}:=\Psi^{-1}(\pp):=\{\mm_1, \mm_2, \dots, \mm_s\} \qquad (\text{as ideals}).$$Then $\fam{M}$ defines an \'etale $S$-rational point 
$$\xibf: A\to S, \quad\text{where} \quad S=\prod_{i=1}^sA/\mm_i.$$We now define 
$$\dd\odot\ee:=\xibf$$ 
and
$$\mathsf{i}_\pp(\dd,\ee):= \dim_{k(\pp)}\!\big(\hat\OO_{\fam{M}}\big/\hat\rho(\ker\xibf)\big),$$ where 
$$\hat\rho: A\to\hat\OO_{\fam{M}}$$ is the versal family of $\fam{M}$. Since $\hat\OO_\fam{M}$ is semi-local the algebra $\hat\OO_{\fam{M}}\big/\hat\rho(\ker\xibf)$ is finite-dimensional over $k(\pp)$, this definition is well-defined. 

The total intersection number is defined in the obvious way:
$$\mathsf{i}(\dd,\ee):=\sum_{\mathfrak{s}\in \dd_\cent\cap \ee_\cent}\mathsf{i}_\pp(\dd,\ee).$$ The definition extends naturally to the case in which $\dd_\cent\cap\ee_\cent\subset\ram(A)$ is more than one point. 
\end{dfn}
Note that we assume here that the intersection happens on a finite fibre, so that the ring $\hat\OO$ is defined. The same definition extends to the generic and analytic fibre in the natural manner.

\begin{remark}Clearly, since the above definitions involve deformation theory, intersections are quite difficult to compute as defined above. However, I feel that the above is the ``correct'' one in the present context. Also it should be said that there are other, more general and abstract, versions of intersection theory on non-commutative spaces (for instance \cite{Jorgensen} in the case of non-commutative surfaces). However, the definition of non-commutative spaces in those versions are global, whereas the approach taken in this paper is fundamentally local. It seems to me that the global approach is not particularly suited for applications i arithmetic. 
\end{remark}

\section{Examples of non-commutative arithmetic spaces}\label{sec:example}

\subsection{Non-commutative quotient spaces}\label{sec:noncomquotient}
Let $X$ be a quasi-projective scheme and $G$ a \emph{finite} group acting on $X$. Since $X$ is quasi-projective and $G$ finite, the quotient $X/G$ exists as a quasi-projective scheme. 

The group $G$ acts on the structure sheaf $\OO_X$ such that $G\cdot \OO_X(U)\subset \OO_X(U)$ and so we can look at the $\OO_X$-algebra $\AA:=\OO_X\langle G\rangle$. This is the finite $\OO_X$-algebra defined as
$$\AA:= \bigoplus_{\tau\in G} \OO_X\cdot\tau, \quad \tau y=\tau(y)\tau, \quad \text{for all $y\in\OO_X$}.$$ The centre of $\AA$ is 
$$\cent(\AA)=\OO_X^G$$ and $\AA$ is finite as a module over $\cent(\AA)$ and hence a PI-algebra over $X/G$. Observe that it is not a PI-algebra over $X$ since $\OO_X^G\subset \OO_X$. 

It is a general fact that the simple $\AA$-modules are in one-to-one correspondence with the set of orbits of $X$ under $G$, in other words, the closed points on $X/G$.  Hence, in this way $\Xscr_\AA$ can be viewed as ``non-commutative thickening'' of $X/G$. 

Also, even if $X$ is not quasi-projective, the quotient $X/G$ exists as a \emph{Deligne--Mumford stack}. If $G$ is not finite the quotient exists as an \emph{Artin} (or \emph{algebraic}) \emph{stack}. Normally, stack quotients are denoted $[X/G]$.  

In view of this, we denote the non-commutative space associated with $\AA$ as
$$\big[\![X/G]\!\big]:=\Xscr_\AA$$ and call this the \emph{non-commutative quotient} of $X$ modulo $G$ and $X/G$ the \emph{coarse space}. Observe that this is commutative. 

\begin{remark}
Let $A$ be a noetherian prime ring which is finite over its centre. Then $A$ is \emph{homologically homogeneous} (\emph{hom-hom}, for short) if $A$ has finite global dimension and, for every pair $\mathfrak{M}_1, \mathfrak{M}_2\in \Max(A)$ such that $\mathfrak{M}_1\cap\cent(A)=\mathfrak{M}_2\cap\cent(A)$, the simple modules $A/\mathfrak{M}_1$ and $A/\mathfrak{M}_2$ have the same projective dimension. Observe that this is a natural extension of regularity in the commutative sense. It is known (see \cite[Thm. 5.6]{StaffordZhang}) that hom-hom implies Auslander-regularity (which we won't define) and in the graded case, Artin--Schelter regularity (which we won't define either).  If $A$ includes a field it is also Cohen--Macaulay. 

Now, a \emph{non-commutative crepant resolution} of a commutative ring $R$ is any ring $\Delta$ such that $\Delta\simeq \End_{R}(M)$ where $M$ is a reflexive $R$-module. Let $V$ be a finite rank free $\oo$-module with a linear action of $G$. Then 
$$R\langle G\rangle \simeq \End_{R^G}(\mathrm{Sym}(V))$$ is a non-commutative crepant resolution of $R^G$. Therefore, $\big[\![X/G]\!\big]$ can be viewed as a non-commutative desingularisation of $R^G$. 
\end{remark}

We will now look a couple of quotient spaces and an example with an order over an arithmetic surface. 
 
\subsection{The plane $\Z/3$-quotient singularity}\label{sec:z3-quotient}
Recall that a point on a non-commutative space $\Xscr_A$ is a representation $\rho: A\to\End(M)$. The point is a closed \'etale point if $\ker\rho$ can be decomposed as $\ker\rho=\mm_1\mm_2\cdots\mm_s$ such that all $A/\mm_i$ are simple algebras (i.e., that the $\mm_i$ are maximal). If $s=1$ we simply say closed point. The algebras $A/\mm_i$ are the underlying points of $\rho$. 

Let $\zeta$ be a primitive third root of unity and assume that $\Z[\zeta]\subseteq\oo$. We let $\mu_3=\langle \sigma\rangle$ act on $\ooo[x, y]$ as
$$\sigma: \quad x\mapsto \zeta x, \quad y\mapsto \zeta^2 y.$$Put 
$$\widehat{A}:=\ooo[x,y]\langle \mu_3\rangle = \frac{\ooo[x,y]\langle \sigma\rangle}{(\sigma x-\zeta x\sigma, \,\,\sigma y - \zeta^2 y\sigma, \,\,\sigma^3=1)}.$$The centre $\cent(\widehat{A})$ is 
$$\widehat{A}^{\mu_3}=\ooo[x^3, xy, y^3]=\ooo[r, s, t]/(t^3-rs), \quad r:=x^3, \,\, s:=y^3, \,\, t:=xy,$$and has a singularity at the origin across all fibres. 

Let $\rhobf:\widehat{A}\to \End_{\ooo}(M)$ be a point on $\widehat{\Xscr}_A$. The point restricts to a point, $\rhobf_{k(\pp)}$, on each fibre for every $\pp\in\Spec(\widehat{A})$. Note, however, that $\widehat{A}_{k(\pp)}$ degenerates to $k(\pp)[x, y, \sigma]$ if there are no non-trivial third roots of unity in $k(\pp)$.

For simplicity of notation, let's fix a prime $\pp\in\Spec(\ooo)$ such that $k:=k(\pp)$ includes a non-trivial third root of unity. Accordingly, we write $A$ instead $\widehat{A}$. Note that $\pp$ need not be a finite prime. 

Now, let $\pf:=(x-a, y-b)$ be a $k$-point on $\Spec(k[x, y])$. The orbit of $\pf$ under $\mu_3$ is the $k$-scheme
$$\Spec\left(\frac{k[x,y]}{(x-a, y-b)}\times\frac{k[x,y]}{(x-\zeta a, y-\zeta^2 b)}\times\frac{k[x,y]}{(x-\zeta^2 a, y-\zeta b)}\right).$$
This corresponds to the $A$-module
$$\rho: A\to\End_k(M_{(a, b)}), \quad M_{(a,b)}:=k\eb_1\oplus k\eb_2\oplus k\eb_3,$$
with actions 
$$\rho(x)=\begin{pmatrix}
 a & 0 & 0\\
 0 & \zeta^{-1} a & 0\\
 0 & 0 & \zeta^{-2} a
\end{pmatrix}, \quad \rho(y)=\begin{pmatrix}
 b & 0 & 0\\
 0 & \zeta^{-2} b & 0\\
 0 & 0 & \zeta^{-1} b
\end{pmatrix}, \quad \rho(\sigma) = \begin{pmatrix}
0 & 0 & 1\\
1 & 0 & 0\\
0 & 1 & 0
\end{pmatrix}$$The kernel of $\rho$ is generated by $\mm:=(x^3-a, y^3-b, \sigma^3-1)$. If $ab\neq 0$ the algebra $A/\mm$ is a central simple algebra and so $\rho$ defines a closed point. On the other hand, if $a=0$ (or $b=0$) the $A/\mm$ is not simple and so $\rho$ is an open point. The underlying closed points are $(x^3-a, y, \sigma^3-1)$ and $(x, y^3-b, \sigma^3-1)$. 

Similary, we take $N_{(u,v)}:=k\eb_1'\oplus k\eb_2'\oplus k\eb_3'$ with
$$\rho'(x)=\begin{pmatrix}
 u & 0 & 0\\
 0 & \zeta^{-1} u & 0\\
 0 & 0 & \zeta^2 u
\end{pmatrix}, \quad \rho'(y)=\begin{pmatrix}
 v & 0 & 0\\
 0 & \zeta^{-2} v & 0\\
 0 & 0 & \zeta^{-1} v
\end{pmatrix}$$with $\rho'(\sigma)=\rho(\sigma)$.

Put
$$\delta(x):=\begin{pmatrix}
	x_{11} & x_{12} & x_{13}\\
	x_{21} & x_{22} & x_{23}\\
	x_{31} & x_{32} & x_{33}
\end{pmatrix}\qquad \delta(y):=\begin{pmatrix}
	y_{11} & y_{12} & y_{13}\\
	y_{21} & y_{22} & y_{23}\\
	y_{31} & y_{32} & y_{33}
\end{pmatrix}$$
and 
$$\delta(\sigma):=\begin{pmatrix}
	s_{11} & s_{12} & s_{13}\\
	s_{21} & s_{22} & s_{23}\\
	s_{31} & s_{32} & s_{33}
\end{pmatrix}$$From the relation $\delta(\sigma^3-1)=0$ follows
\begin{align}
\begin{split}
	s_{11} &= -s_{22}-s_{33}\\
	s_{12} &= -s_{23}-s_{31}\\
	s_{13} &= -s_{21}-s_{32}
\end{split}
\end{align}

Since $x$ and $y$ commutes we find that $\delta(xy-yx)=0$ leads to (after some simplifications)

\begin{align}\label{eq:xy-system}
\begin{split}
(b  - v) x_{11} &=(a - u) y_{11}\\
 (\zeta b - v) x_{12} &= (\zeta^{-1} a - u) y_{12}\\
 (\zeta^{-1}b - v) x_{13} &=(\zeta a - u) y_{13} \\
(b  - \zeta v) x_{21} &= (a - \zeta^{-1} u) y_{21} \\
 \zeta^2 (b - v) x_{22} &=  (a  - u )y_{22}\\
 ( b  - \zeta^2 v) x_{23} &= (\zeta^2 a -u) y_{23} \\
(b - \zeta^{-1} v) x_{31} &= (a - \zeta u) y_{31}\\
  (\zeta^2 b  -v) x_{32} &= ( a - \zeta^2 u) y_{32}\\
   (b - v)x_{33} &= \zeta^2( a- u) y_{33}
\end{split}
\end{align}

Similarly, $\delta(\sigma x-\zeta x\sigma)=0$ gives, (again after simplifications)

\begin{align*}
x_{11}&=\zeta x_{22} -(a-u)s_{21}\\
x_{31}&=\zeta x_{12}-(a-\zeta u)(s_{22}+s_{33})\\
x_{32}&=\zeta x_{13}-\zeta(\zeta a-u)(s_{23}+s_{31})\\
x_{33}&=\zeta x_{22}-(a-u)(s_{21}+s_{32})\\
x_{21}&=\zeta^2 x_{13}-(a-\zeta^2 u)s_{23}\\
x_{23}&=\zeta^2 x_{12}-\zeta(a-\zeta u)s_{22}
\end{align*}
and, by symmetry, $\delta(\sigma y-\zeta^2 y \sigma)=0$, 
\begin{align}\label{eq:y-system}
\begin{split}
y_{11}&=\zeta y_{22} -(b-v)s_{21}\\
y_{31}&=\zeta y_{12}-(b-\zeta v)(s_{22}+s_{33})\\
y_{32}&=\zeta y_{13}-\zeta(\zeta b - v)(s_{23}+s_{31})\\
y_{33}&=\zeta y_{22}-(b-v)(s_{21}+s_{32})\\
y_{21}&=\zeta^2 y_{13}-(b-\zeta^2 v)s_{23}\\
y_{23}&=\zeta^2 y_{12}-\zeta(b-\zeta v)s_{22}.
\end{split}
\end{align}

We find that we can chose $x_{12}$, $x_{13}$ and $x_{22}$ as parameters from $\delta(x)$ and $s_{21}$, $s_{22}$, $s_{23}$, $s_{31}$, $s_{32}$ and  $s_{33}$ from $\delta(\sigma)$.  For a generic point $(u,v)=(a,b)$ (which is $(a^3, b^3, ab)$ in $\Spec(\cent(A))$), we can additionally choose $x_{11}$ and $y_{11}$ as parameters. 

Since $\sigma$ only scrambles the entries in a matrix upon multiplication (from the left and right), we easily see that the dimension of the inner derivations is 9. Therefore, for a generic point $(u, v)=(a,b)$ the dimension is two as it should be.

However, for specific choices of points, the ext-dimensions are higher. These are the open points defined above. In fact, the images of the coordinate axes from $\mathbb{A}^2$ to $\Spec(\cent(A))$ is $\ram(A)$:
$$\ram(A)=\Spec(k[r, s, t]/(t^3-rs, s, t))\cup\Spec(k[r, s, t]/(t^3-rs, r, t)).$$ 

Indeed, put $b=v=0$ and, say, $u=\zeta a$. Then $u^3=a^3$ so both $(a, 0)$ and $(\zeta a, 0)$ lie over the same point in $\cent(A)$. We see that the left-hand side of (\ref{eq:xy-system}) is zero. Also, in row 3,  we find $\zeta(a-a)y_{13}=0$. This implies that $y_{13}$ is a free parameter. Similarly, it looks like $y_{21}$ and $y_{32}$ would also become free, but from (\ref{eq:y-system}) both of these are expressible in terms of $y_{13}$.  Hence we gain one free parameter and so 
$$\Ext^1_A\big(M_{(a,0)}, N_{(\zeta a, 0)}\big) = k.$$We easily see that there are 1-dimensional Ext's between all three points above $(a^3,0)$. By symmetry we find that the same holds for the ``$y$-axis'' $(0, b^3)$. 

It is quite easy to convince oneself that all deformations are unobstructed. 

Put $M_i:=M_{(\zeta^{i-1}a, 0)}$, $i=1, 2, 3$. For $\pf_\cent\in\ram(A)$, we can view the fibre $\phi^{-1}(\pf_\cent)=\{M_1, M_2, M_3\}$ as an \'etale point. Indeed, let $\pf_\cent$ correspond to the maximal ideal $\mm$. Then the extension of $\mm_\cent$ to $A$ splits into three maximal ideals $\mm_1, \mm_2,\mm_3$, corresponding to $M, N, P$, and so $\rho: A\to \End_k(A/\mm_\cent)$ is an \'etale point with $\ker\rho=\mm_1\mm_2\mm_3$. The underlying points are then the simple algebras $A/\mm_i$. 

We summarise the discussion with the following theorem.

\begin{thm}The non-commutative space $\big[\![\mathbb{A}^2_\oo/\mu_3]\!\big]$ is 

$$\big[\![\mathbb{A}^2_\oo/\mu_3]\!\big]=(\Mod(A), \OO),$$where, for $\pf\in\azu(A)$, viewed as both a module and point,
$$\hat \OO_\pf=\hat\OO_{\mathbb{A}^2_\oo/\mu_3, \pf}\simeq \End_k(\pf)\otimes_k \hat H_\pf.$$Over $\ram(A)$, we have
the \'etale point $\fam{M}:=\{M_1, M_2, M_3\}$ and so
$$\hat H_{\fam{M}}=\begin{pmatrix}
	k\langle\!\langle t^1_{11}, t^2_{11}\rangle\!\rangle, & \langle t_{12}\rangle  & \langle t_{13}\rangle \\
	\langle t_{21}\rangle & k\langle\!\langle t^1_{22}, t^2_{22}\rangle\!\rangle & \langle t_{23}\rangle \\
	\langle t_{31} \rangle &  \langle t_{33}\rangle  & k\langle\!\langle t^1_{33}, t^2_{33}\rangle\!\rangle
\end{pmatrix},$$ and, with hopefully clear notation,
$$
\hat \OO_{\fam{M}}= \Hom_k(\fam{M})\otimes_k \hat H_\fam{M}.
$$
The closed points of $\big[\![\mathbb{A}^2_\oo/\mu_3]\!\big]$ are stratified as 
$$\big[\![\mathbb{A}^2_\oo/\mu_3]\!\big]_n = (\Mod^\bullet_n, \OO_n), \quad \text{where $n=1, 3$},$$and
$$\big[\![\mathbb{A}^2_\oo/\mu_3]\!\big]_3=\azu(A), \quad\text{and}\quad\big[\![\mathbb{A}^2_\oo/\mu_3]\!\big]_1=\ram(A).$$
\end{thm}
It is important to observe that the statements made in the theorem are made fibre-by-fibre (anything else is meaningless since everything is trivial across different characteristics). We have chosen not to make this explicit with an awkward notation such as $\big[\![\mathbb{A}^2_\oo/\mu_3]\!\big]\otimes k(\pp)$ or something similar.
\begin{figure}
\centering
\includegraphics[scale=.7]{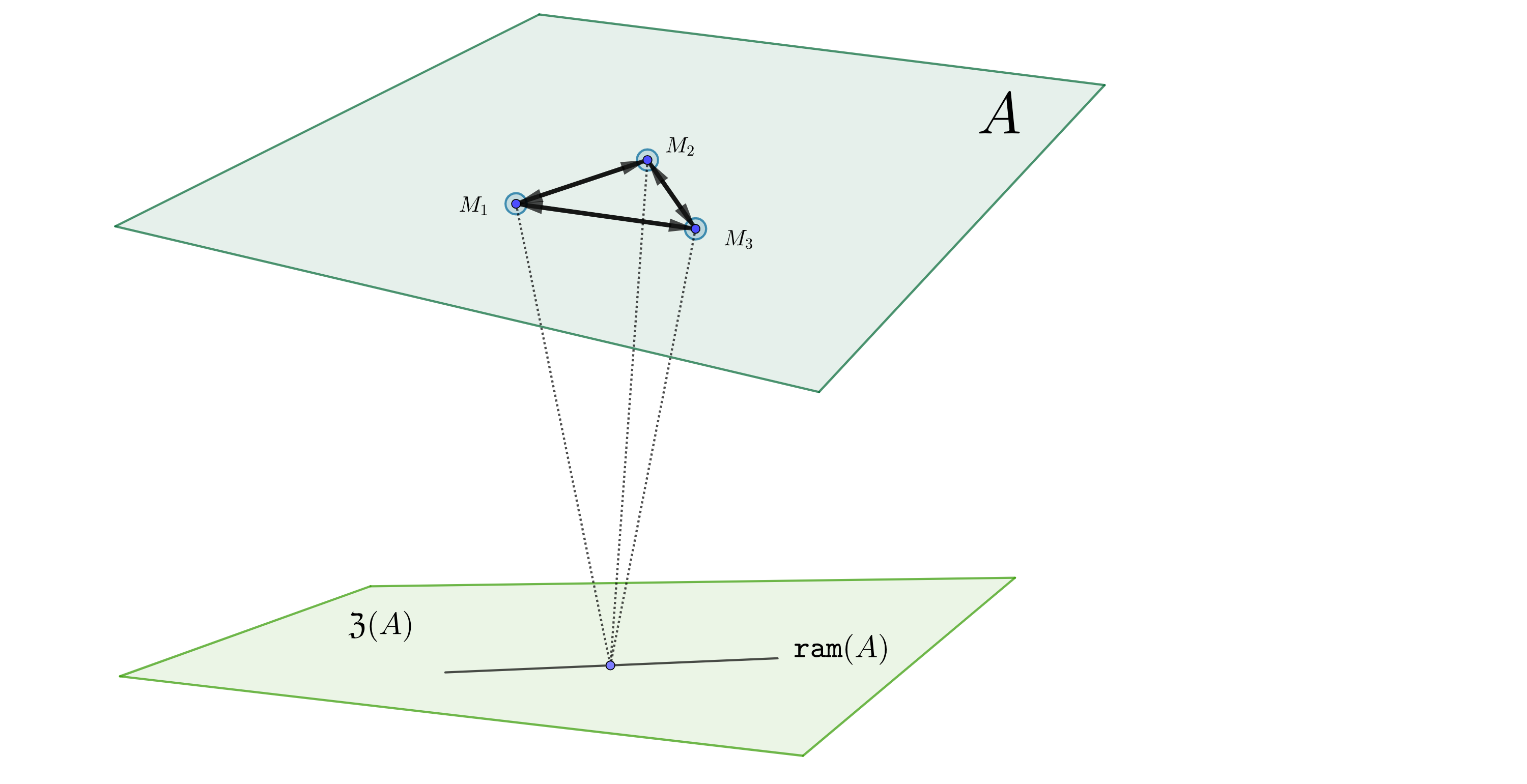}
\caption{Ramification situation for $\big[\![\mathbb{A}_\oo^2/\mu_3]\!\big]$ on a fibre.}
\end{figure}

\subsubsection{Arithmetic geometry of $\big[\![\mathbb{A}_\oo^2/\mu_3]\!\big]$}

Recall that an \'etale rational point on $\Xscr_A$ is an algebra morphism $\xibf: A\to S$ such that 
$$A/\ker\xibf=\prod_i^s A/\mm_i,$$ is a direct product of prime algebras. If $S$ is artinian this implies that the $\mm_i$ are all maximal so the $\bar\xi_i:=A/\mm_i$ are then simple algebras. These are the underlying points of $\xibf$. 

Let $\dd_\cent$, $\ee_\cent$ and $\ff_\cent$ be the central divisors
 $$\dd_\cent:=\{r=a^3, s=t=0\}, \quad\ee_\cent=\{r=t=a^3, s=1\},$$and $$\ff_\cent:=\{r=1, s= t=b^3\}.$$ Observe that $\dd_\cent\subset\ram(A)$ but $\ee_\cent, \ff_\cent\subset\azu(A)$. 
 
The intersection between $\ee_\cent$ and $\ff_3$ is $\ee_\cent\cap\ff_\cent=\{r=s=t=1\}$ corresponding to the ideal $$\mm:=(r-1, s-1, t-1)\subset\cent(A).$$This gives the rational point
$$\ee\odot\ff=\Big(\xi:A\to A/\mm=\frac{A}{(x^3-1, y^3-1)}\Big)$$which is a cental simple algebra (as it should since the intersection is in $\azu(A)$). 

The intersection $\dd_\cent\cap\ff_\cent$ is inside $\ram(A)$ and corresponds to the ideal
$$\mathfrak{n}:=(r-1, s, t)=(x^3-1, y^3, xy).$$This defines the \'etale rational point
$$\dd\odot\ff=\Big(\xibf: A\longrightarrow\prod_{j=0}^2 \frac{A}{A(x-\zeta^j, y)A}\Big).$$Observe that 
$$\sigma(x-\zeta^j)\zeta^i x\sigma\in A(x-\zeta^j, y)A$$ for all $i$. This implies (taking $i=-j$) that $\sigma x-x\sigma=0$ in $A/A(x-\zeta^j, y)A$, implying that $A/A(x-\zeta^j, y)A$ is actually commutative. In fact, as is easily seen, $A/A(x-\zeta^j, y)A=k$. Therefore,
$$\dd\odot\ff=\Big(\xibf: A\longrightarrow k\times k\times k\Big).$$

We leave for the reader to compute the intersection numbers (which is not quite so easy).

The module $M_{(\alpha,\beta)}$, with $\alpha, \beta \in k'$, defines a $k'$-rational point on $\big[\![\mathbb{A}_\oo^2/\mu_3]\!\big]$  via the structure morphism $$\xi:A\to\End_{k'}(M_{(\alpha,\beta)}).$$  

In addition, let $$\xi_\cent: k[r,s,t]/(t^3-rs)\longrightarrow k'$$ be a $k'$-rational point on $\Spec(\cent(A))$. Then 
$$\xibf: A\longrightarrow A/\ker\rho_\cent$$ is an \'etale $(A/\ker\rho_\cent)$-rational point on $\big[\![\mathbb{A}_\oo^2/\mu_3]\!\big]$. Note that $A/\ker\rho_\cent$ is a $k'$-algebra. 

\begin{example}
Let $\xi_\cent$ be the point corresponding to the ideal $$\ker\xi_\cent:=(r-\zeta, s-1, t-\tau), \quad \tau^3=\zeta.$$Then $k'=k(\sqrt[3]{\zeta})$. Note that $\ker\xi_\cent\in \azu(A)(k')$ and the lift of $\xi_\cent$ to $A$ is the rational point
$$\xi: A\to A/\ker\xi_\cent=\frac{A}{\left( x^3-\zeta, y^3-1\right)}.$$This is a $k'$-central simple algebra.

Suppose now that $\xi_\cent$ is the point corresponding to the ideal 
$$\ker\xi_\cent:=(r-\zeta, s, t)\in \ram(A)(k').$$ Then the lift of $\xi_\cent$ to $A$ is 
$$\xibf: A\to A/\ker\xi_\cent=\frac{A}{\left( x^3-\zeta, y\right)}.$$This defines an \'etale rational point with underlying points $\bar\xi_i$, the points corresponding to the ideals over $\ker\xi_\cent$ in $A$. Note that the residue ring of $\xibf$ is the \'etale algebra $E_{\xibf}=k'\times k'\times k'$.
\end{example}

Let us finally compute the non-commutative height of a point in $\ram(A)$. The adjacency matrix is 
$$E=\begin{pmatrix}
2 & 1 & 1\\
1 & 2 & 1\\
1 & 1 & 2
\end{pmatrix}$$whose eigenvalues are $\{4, 1, 1\}$. Therefore $H_{k'}^\mathtt{nc}(\xibf)=4$. 

Observe that $H_{k'}^\mathtt{nc}(\xibf)$ is only dependent on the local data $\hat H_\fam{M}$. The other coordinates in the height vector 
$$H_{k', \alpha}^\mathtt{tot}(\xibf)=\Big(H_\alpha^\cent(\xibf), H_{k'}^\mathtt{rep}(\xibf),H_{k'}^\mathtt{nc}(\xibf)\Big)$$are the entities directly related to the coordinates of the point $\xibf$. Unfortunately I don't know how to compute this even in the simplest (non-trivial) case. 
\subsection{The non-commutative thickening $\Spec(\Z)^{\mathrm{nc}}$}

Let $d\equiv 2,3\,\,\text{mod}\,\, 4$ and square-free. Put $F:=\Q[\sqrt{d}]$ and $\oo_F$ its ring of integers. The stated hypothesis on $d$ ensures that $\oo_F=\Z[\sqrt{d}]$, with discriminant $\delta_F=4d$. We use the presentation $\oo_F=\Z[x]/(x^2-d)$. 

This gives a $\Z/2$-cover $\phi:\Spec(\oo_F)\to \Spec(\Z)$, ramified over $2$ and the prime divisors of $d$. Since $|\Z/2|=2$, we see that the cover is wildly ramified over $2$ and tamely ramified for all other ramification points. 

Put $\mathfrak{H}:=\Z/2$. The orbits of $\mathfrak{H}$ on $\oo_F$ come in three types: 
\begin{itemize}
	\item[(i)] The number of points in the orbit is two. This is the (completely) split case. 
	\item[(ii)] The number of points in the orbit is one, without multiplicity. This is the inert case.
	\item[(iii)] The number of points in the orbit is one, with multiplicity. This is then the ramified case. 
\end{itemize} Observe that ``point'' means \emph{closed} point and that $\Spec(\oo_F)\big/\mathfrak{H}=\Spec(\Z)$.  

Accordingly, the corresponding $\oo_F\langle\Gamma\rangle$-modules look very different. Let $\tau$ denote the non-trivial element of $\mathfrak{H}$, thus acting as $\tau(a+b\sqrt{d})=a-b\sqrt{d}$. We look at the different cases in turn. Let $\pp\in\Spec(\oo_F)$ be a prime over $p\in\Spec(\Z)$ and put 
$$A:=\oo_F\langle\mathfrak{H}\rangle\simeq \frac{\Z\langle x,\tau\rangle}{\left(x^2-d,\tau x+x\tau,\tau^2-1\right)}.$$

\subsubsection{The split case}
Assume that $p$ is split. Hence $(p)=\pp_+\pp_-$ and the orbit of $\pp_+$ (or $\pp_-$, of course) is $\boldsymbol{orb}(p):=\{\pp_+,\pp_-\}$. Since the orbits are precisely the fibres of the covering morphism $\phi$, the orbit is completely determined by the underlying prime $p$ and we parametrize the orbits using the quotient $\Spec(\Z)$. 

The $A$-module corresponding to $\boldsymbol{orb}(p)$ is
$$M=\frac{\F_p[x]}{(x-a)}\eb_1\oplus\frac{\F_p[x]}{(x+a)}\eb_2,$$where $x^2-d=(x-a)(x+a)$ modulo $p$. In other words,
$$M=\F_p\eb_1\oplus\F_p\eb_2,\quad\text{with}\quad x\mapsto \begin{pmatrix}
	a & 0\\
	0 & -a
\end{pmatrix},\quad \tau\mapsto\begin{pmatrix}
	0 & 1\\
	1 & 0
\end{pmatrix}.$$This defines a simple $A$-module. 

Any $\delta\in\Der_{\Z}\!\!\left(A,\End_{\F_p}(M)\right)$ must satisfy
$$\delta(x^2-d)=\delta(x^2)=0,\quad \delta(\tau x+x\tau)=0 \quad\text{and}\quad \delta(\tau^2-1)=\delta(\tau^2)=0.$$The first relation leads to 
\begin{align*}
	\delta(x^2)=
		\begin{pmatrix}
	x_{11} & x_{12}\\
	x_{21} & x_{22}
	\end{pmatrix}\begin{pmatrix}
			a & 0\\
	0 & -a
	\end{pmatrix}+\begin{pmatrix}
			a & 0\\
	0 & -a
	\end{pmatrix}\begin{pmatrix}
	x_{11} & x_{12}\\
	x_{21} & x_{22}
	\end{pmatrix}=\begin{pmatrix}
		2ax_{11} & 0\\
		0 & -2ax_{22}
	\end{pmatrix}
\end{align*}implying that $x_{11}=x_{22}=0$ (since $p\neq 2$). The second equation gives
$$\delta(\tau x+x\tau)=\begin{pmatrix}
	2ad_{11}+x_{21}+x_{12} & 0\\
	0 & -2ad_{22}+x_{21}+x_{12}
\end{pmatrix}=\mathbf{0},$$where we have used that $x_{11}=x_{22}=0$. Similarly, 
$$\delta(\tau^2)=\begin{pmatrix}
	d_{12}+d_{21} & d_{11}+d_{22}\\
	d_{11}+d_{22} & d_{12}+d_{21}
\end{pmatrix}=\mathbf{0},$$implying that $d_{22}=-d_{11}$ and $d_{21}=-d_{12}$. We see that $d_{11}$ can be expressed in terms of $x_{12}$ and $x_{21}$ so $d_{11}$ and $d_{22}$ are determined when $x_{12}$ and $x_{21}$ are fixed. Therefore, we can choose $d_{12}$, $x_{12}$ and $x_{21}$ as free parameters. A small computation shows that all derivations are inner, i.e., $\dim_{\F_p}(\mathrm{Ad})=3$, and so 
$\Ext^1_A(M,M)=0$ in the completely split case. Observe that this is relative to $\Z$. Therefore, $\hat H=k$ and so
$$\hat\OO_{\{M\}}=\hat H\otimes_k\End_k(M)=\End_k(M). $$
This means that $M$, as an $A$-module, is rigid in the deformation-theoretic sense which is certainly reasonable (one cannot ``deform'' prime numbers). This should certainly apply to the inert case also. Let's show explicitly that this is true.

\subsubsection{The inert case}
When $p$ is inert, we have $(p)=\pp\in\Spec(\oo_F)$. This means that $x^2-d$ is irreducible modulo $p$. However, even though there is only one point in the orbit, the corresponding module is 2-dimensional. The orbit is the fibre of $\phi$ so the corresponding module is
$$M=\oo_F\otimes_{\Z}k(p)=\oo_F/(p)=\frac{\Z[x]/(x^2-d)}{(p)}=\F_p[x]/(x^2-d)=\F_p\eb_1\oplus\F_p\eb_2.$$The action of $x$ is given as
$$x\cdot\eb_1=\eb_2,\quad\text{and}\quad x\cdot\eb_2=d\eb_1.$$The action of $\tau$ is slightly trickier. Observe first that it cannot be the identity since $p$ is not ramified. However, $M$ is a quadratic extension of $\F_p$ and $\tau$ reduces to the non-trivial extension of this extension so must be given by $\tau(\eb_1)=\eb_1$ and $\tau(\eb_2)=-\eb_2$ modulo $p$. On matrix form we thus have
$$x\mapsto \begin{pmatrix}
	0 & d\\
	1 & 0
\end{pmatrix},\quad\text{and}\quad \tau\mapsto \begin{pmatrix}
	1 & 0\\
	0 & -1
\end{pmatrix}.$$ This is a simple $A$-module. Notice that the matrix corresponding to $x$ does not have any eigenvectors over $\F_p$ (since $x^2-d$ is irreducible) so $x$ cannot preserve any 1-dimensional subspace of $M$. 

The same type of calculation as in the previous case shows that $M$ is rigid here also, i.e., $\Ext^1_{A}(M,M)=0$ (implying that $\hat H=k$), hence
$$\hat\OO_{\{M\}}=\End_k(M)$$in the inert case also.
\subsection{The ramified case}
Suppose now that $p\mid d$. Then 
$$M_3:=\oo_F/(p)=\frac{\Z[x]/(x^2-d)}{(p)}=\F_p[x]/(x^2)=\F_p\eb_1\oplus\F_p\eb_2,$$ with $x\cdot\eb_1=\eb_2$ and $x\cdot\eb_2=0$. The induced action of $\tau$ on $M_3$ is $\tau(\eb_1)=\eb_1$, $\tau(\eb_2)=\eb_2$. This is an indecomposable module, but not simple. 

The composition series $\F_p\eb_2\subset \F_p\eb_1\oplus\F_p\eb_2$ gives the two simple modules
$$M_1:=\F_p\eb_2, \quad M_2:= (\F_p\eb_1\oplus\F_p\eb_2)/\F_p\eb_2\simeq \F_p\eb$$where $x\eb=0$. Observe that $M_1\simeq M_2$. This should be interpreted as $M_3$ including two isomorphic, but distinct, points (i.e., modules). 

Let $\delta: A\to \Hom_{\F_p}(M_1, M_2)=\End_{\F_p}(M_1)$ be a derivation. We find
$$\delta(x^2)\eb=\delta(x)x\eb+x\delta(x)\eb=0, \qquad \delta(\tau^2)\eb=\delta(\tau)\tau\eb+\tau\delta(\tau)\eb=2d_\tau\eb,$$the first one following since $x\eb=0$. Take $\theta\in\End_{\F_p}(M_1)$. Then,
$$(\theta x-x\theta)\eb=0,\quad\text{and}\quad (\theta\tau-\tau\theta)\eb=0,$$ so $\dim(\mathrm{Ad})=0$. Consequently, 
$$\Ext^1_A(M_1, M_2)=\Ext_A^1(M_1, M_1)=\begin{cases}
k, & \text{if $p\neq 2$, and}\\
k^2, & \text{if $p= 2$.}\end{cases}$$

Now, let $\delta: A\to \Hom_{\F_p}(M_2, M_3)$. Put $\delta(x)\eb = d_1\eb_1+d_2\eb_2$ and $\delta(\tau)\eb = t_1\eb_1+t_2\eb_2$. Then we find
$$\delta(x^2)\eb=d_1\eb_2, \qquad \delta(\tau^2)\eb=2t_1\eb_1+2t_2\eb_2.$$

We can thus choose $d_2$ and $t_2$ as free parameters if $p\neq 2$ and, additionally, $t_2$ as free when $p= 2$. Computing the inner derivations we find that these are 1-dimensional so 
$$\Ext^1_A(M_2, M_3)=\begin{cases}
k, & \text{if $p\neq 2$, and}\\
k^2, & \text{if $p= 2$.}\end{cases}$$

In the other direction, we can compute 
$$\Ext^1_A(M_3, M_2)=\begin{cases}
0, & \text{if $p\neq 2$, and}\\
k^2, & \text{if $p= 2$.}\end{cases}$$Finally, we compute $\Ext^1_A(M_3, M_3)$. 

We have
$$\rho(x)=\begin{pmatrix}
	0 & 0\\
	1 & 0
\end{pmatrix}, \quad \text{and}\quad 
\rho(\tau)=\begin{pmatrix}
	1 & 0 \\ 
	0 & 1
\end{pmatrix}.$$ For $\delta\in\Der_{\F_p}(A, \End_{\F_p}(M_3))$, put 
$$\delta(x)=\begin{pmatrix}
	x_{11} & x_{12}\\
	x_{21} & x_{22}
\end{pmatrix}, \quad\text{and}\quad \delta(\tau)=\begin{pmatrix}
	s_{11} & s_{12}\\
	s_{21} & s_{22}
\end{pmatrix}.$$From $\delta(x^2)=0$, we find that $x_{12}=0$, $x_{22}=-x_{11}$ and $x_{21}$ free; from $\delta(\tau^2)=0$ we get 
$$2s_{11}=2s_{12}=2s_{21}=2s_{22}=0.$$Assume first, that $p\neq 2$. 

Then $s_{11}=s_{12}=s_{21}=s_{22}=0$ and a general derivation can be written as
$$\delta = \begin{pmatrix}
	d_{11} & 0\\
	d_{21} & -d_{11}
\end{pmatrix}.$$Therefore, an inner derivation comes from an element $\theta\in\End_{\F_p}(M_3)$ on the form
$$\theta = \begin{pmatrix}
	\theta_{11} & 0\\
	\theta_{21} & -\theta_{11}
\end{pmatrix}.$$Computing $\theta x-x\theta$ we find that $2\theta_{11}=0$, and so $\dim(\mathrm{Ad})=1$. Consequently, $\Ext_A^1(M_3, M_3)=\F_p$. 

On the other hand, in the wildly ramified prime $p=2$, we find $\dim(\mathrm{Ad})=0$, implying that $\Ext_A^1(M_3, M_3)=\F_2^2$.

Since all fibres of $\Z\to A$ are central simple, $A$ is Azumaya over $\cent(A)=\Z$.

\subsubsection{The space $\big[\![\Z[\sqrt{d}]\big/\Gamma]\!\big]$}
\begin{thm}
The space 
$$\big[\![\Z[\sqrt{d}]\big/\Gamma]\!\big]=\Big(\Mod\big(\Z[\sqrt{d}]\big/\Gamma\big), \boldsymbol{\OO}\Big)$$is an Azumaya thickening of $\Z$. If $M_{/\F_p}$ is an unramified point
$$\hat\OO_{\{M\}}=\End_{\F_p}(M).$$ In the tamely ramified case we have, with $\fam{M}:=\{M_1, M_2, M_3\}$,
$$\hat H_{\fam{M}}=\begin{pmatrix}
	\F_p[[t_{11}]] & \langle t_{12}\rangle & \langle t_{13}\rangle\\
	\langle t_{21}\rangle & \F_p[[t_{22}]] & \langle t_{23}\rangle \\
	0 & 0 & \F_p[[t_{33}]]
\end{pmatrix}$$In the wildly ramified case ($p= 2$), we have
$$\hat H_{\fam{M}}=\begin{pmatrix}
	\frac{\F_2\langle\!\langle u_1,u_2\rangle\!\rangle}{\left(u_1^2, \,u_2^2\right)} & \langle t_{12}^1, t_{12}^2\rangle & \langle t_{13}^1, t_{13}^2\rangle\\
	\langle t_{21}^1, t_{21}^2\rangle & \frac{\F_2\langle\!\langle v_1,v_2\rangle\!\rangle}{\left(v_1^2, \,v_2^2\right)} & \langle t_{23}^1, t_{23}^2\rangle \\
	0 & 0 &  \frac{\F_2\langle\!\langle w_1,w_2\rangle\!\rangle}{\left(w_1^2, \,w_2^2\right)}
\end{pmatrix}$$
In both cases the versal family is
$$\hat\OO_\fam{M}=\Hom_{\F_p}(\fam{M})\otimes_{\F_p}\hat H_\fam{M}.$$
\end{thm}
\begin{proof}
The only thing not proven in the discussion above are the obstructions in the wild ramification points. We omit this computation. 
\end{proof}
\subsection{Orders over a curve}

Let $Y:=\Spec(R)$, with $R:=\ooo[u,v]/(f(u,v))$, be an arithmetic surface over $\ooo$ such that $\zeta=\zeta_3\in\oo$ is a (primitive) third root of unity. Then 
$$A:=\frac{R\langle x,y\rangle }{(xy-\zeta yx, x^3 - u, y^3-v)}=\frac{\ooo[u,v]\langle x,y\rangle }{\left(f(u,v), xy-\zeta yx, x^3 - u, y^3-v\right)}$$ is an algebra over $Y$ with central scheme $Y=\Spec(\cent(A))$ itself. 

Let $M'$ be the $A$-module $M':=k\eb$ with actions
$$u\eb = \alpha \eb, \quad v\eb = \beta \eb, \quad x\eb = a\eb, \quad y\eb = b\eb, \quad \alpha, \beta, a, b\in k.$$Note that $f(\alpha, \beta)=0$. This is an $A$-module over the closed point $(u-\alpha, v-\beta)$, $\alpha, \beta\in k=k(\pp)$, where $\pp\in\Spec(\oo)$.  

We have
$$u\eb=\alpha \eb, \,\, x\eb = a\eb \quad \Longrightarrow \quad x^3\eb = u\eb=\alpha \eb$$so $a^3=\alpha$.  Therefore, there are three possibilities for $a$, namely, $a=\sqrt[3]{\alpha}$, $a_1:=\zeta\sqrt[3]{\alpha}=\zeta a$ and $a_2:=\zeta^2\sqrt[3]{\alpha}=\zeta^2 a$. The same applies to $v$, $y$, $\beta$ and $b$.

For $M'$ to be an $A$-module we need to have that $(xy-\zeta yx)\eb=0$:
$$(xy-\zeta yx)\eb= ab-\zeta ab=(1-\zeta)ab=0,$$hence $ab=0$. Assume that $b=0$, implying that $\beta=0$.

Take two $A$-modules
$$M:=k\eb, \quad u\eb = \alpha \eb, \,\, x\eb = a\eb$$ 
and $$N:=k\fb, \quad u\fb = \alpha \fb, \,\, x\fb = \zeta a\fb.$$ Let $\delta$ be a derivation $\delta: A\to\Hom(M, N)$. Put $\delta(x)\eb:=d_x \fb$, and, in addition $\delta(u)\eb := d_u\fb$. We have
\begin{align*}
\delta(x^3-u)\eb&=\big(\delta(x)x^2+x\delta(x)x+x^2\delta(x)-\delta(u)\big)\eb\\
&=(a^2 d_x + \zeta a^2 d_x +\zeta^2 a^2 d_x-d_u)\fb\\
&=(1+\zeta+\zeta^2)a^2d_x\fb-d_u\fb\\
&=-d_u\fb.
\end{align*}Since $\zeta$ is a third root of unity $1+\zeta+\zeta^2=0$. Hence, $d_u=0$ and $d_x$ can be chosen to be a free parameter. 

Also, putting $\delta(v):=d_v$ and remembering that $u\eb=\alpha \eb$ and $v\eb=0$,
$$\delta(uv-vu)\eb= \big(\delta(u)v+u\delta(v)-\delta(v)u-v\delta(u)\big)\eb = 0.$$ Therefore, we can choose $d_v$ as a free parameter. 

Furthermore, we need to have $\delta(xy-\zeta yx)\eb=0$ so 
$$\delta(xy-\zeta yx)\eb = \big(\delta(x)y+x\delta(y)-\zeta \delta(y)x-\zeta y\delta(x)\big)\eb=\zeta a d_y\fb-\zeta a d_y\fb=0 ,$$ where $d_y:=\delta(y)$, can be chosen as a free parameter.  

So far we have $d_x$, $d_v$ and $d_y$ as free parameters. 

Let $\theta\in\Hom(M, N)$ with $\theta\eb = t\fb$. We directly see that $\dim(\mathrm{Ad})=1$ since 
$$(\theta x- x\theta)\eb=\theta x\eb -x\theta\eb=(t_1a-\zeta t_1a)\fb=(1-\zeta)at_1\fb.$$

Therefore,
$$\Ext^1_A(M, N)=k^2.$$

On the other hand, choose $N'$ as the module where $x$ acts as $x\eb = a_2\eb$. We then get 
\begin{align*}
\delta(x^3-u)\eb&=\big(\delta(x)x^2+x\delta(x)x+x^2\delta(x)-\delta(u)\big)\eb\\
&=(a^2 d_x + \zeta^2 a^2 d_x +\zeta^4 a^2 d_x-d_u)\fb\\
&=(1+\zeta^2+\zeta^4)a^2d_x\fb-d_u\fb\\
&=-d_u\fb,
\end{align*}since $\zeta^4=\zeta$. Hence $d_x$ is still free and $d_u=0$. We also, still, have that $d_v$ is free. However,
\begin{align*}
\delta(xy-\zeta yx)\eb &= \big(\delta(x)y+x\delta(y)-\zeta \delta(y)x-\zeta y\delta(x)\big)\eb\\
&=(a_2d_y-\zeta ad_y)\fb\\
&=\zeta(\zeta-1)ad_y\fb.
\end{align*}Therefore, $d_y=0$. The inner derivations are clearly still 1-dimensional. This means that
$$\Ext^1_A(M, N')=k.$$

Shifting the points cyclically we find that the diagram must look like the depiction in figure \ref{fig:3curvetangent}.
\begin{figure}[h!]
\centering
\includegraphics[scale=0.7]{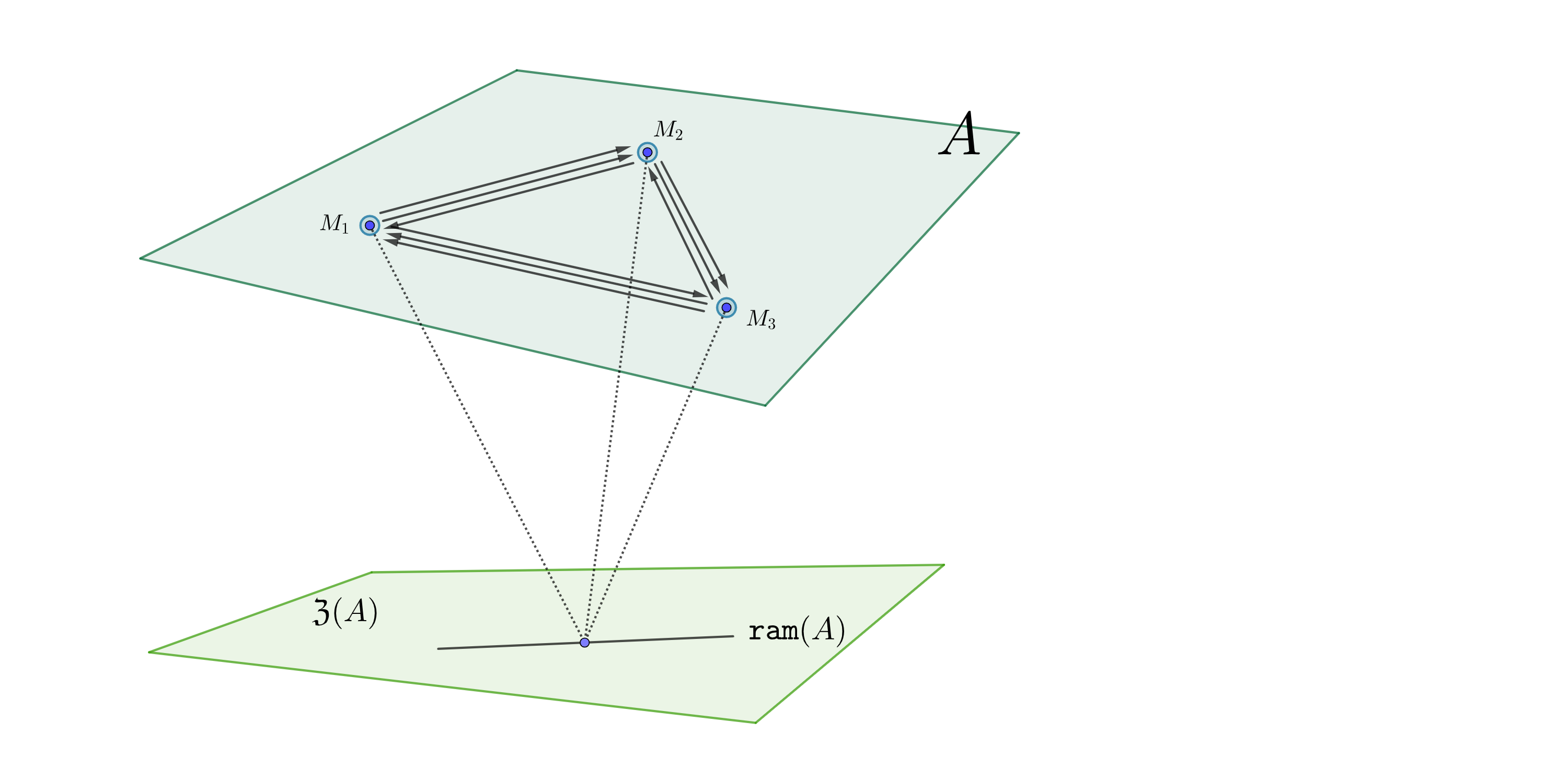}
\caption{Tangent situation}\label{fig:3curvetangent}
\end{figure}

We easily find that 
$$\Ext^1_A(M,M)=\begin{cases}
	1, & \mathrm{char}(k)\neq 3\\
	2, &  \mathrm{char}(k)= 3.
	\end{cases}$$
The adjacency matrix becomes when $\mathrm{char}(k)\neq 3$,
$$E=\begin{pmatrix}
	1 & 2 & 1\\
	1 & 1 & 2\\
	2 & 1 & 1
	\end{pmatrix}$$ and, when  $\mathrm{char}(k)= 3$,
	$$E=\begin{pmatrix}
	2 & 2 & 1\\
	1 & 2 & 2\\
	2 & 1 & 2
	\end{pmatrix}$$
In the first case the characteristic polynomial is $P(\lambda) = \lambda^3 - 3\lambda^2 - 3\lambda - 4$ and in the second $P(\lambda) = \lambda^3 - 6\lambda^2 -6 \lambda - 5$. Hence, we find that the non-commutative height is dependent on the characteristic of the ground field. 

We leave for the reader to play around with rational points, divisors and intersection theory and report back to the author when finished. 
\begin{center}
\rule{0.50\textwidth}{0.5pt}
\end{center}
\bibliographystyle{alpha}

\bibliography{ref_Arit_NC}
%
%
\end{document}